\documentclass[11pt]{amsart}

\usepackage{amsfonts, amssymb, amscd}
\usepackage{upgreek}

\usepackage{verbatim}
\usepackage{amssymb}
\usepackage{mathrsfs}
\usepackage{graphicx}
\usepackage{bm}
\usepackage[all]{xy}
\usepackage{tikz}
\usepackage{subcaption}
\usepackage{subfiles}
\usepackage[toc,page]{appendix}
\usepackage{mathtools}
\usepackage{comment}
\usepackage{enumerate}
\usepackage{enumitem}

\usepackage{graphicx}
\graphicspath{ {images/} }

\usepackage{hyperref}
\hypersetup{colorlinks, linkcolor={violet}, citecolor={violet}}

\newcommand{\Zz}{\mathbb{Z}}
\newcommand{\Cc}{\mathbb{C}}
\newcommand{\Pp}{\mathbb{P}}
\newcommand{\Rr}{\mathbb{R}}
\newcommand{\Qq}{\mathbb{Q}}
\newcommand{\Nn}{\mathbb{N}}

\newcommand{\ep}{\epsilon}
\newcommand{\dto}{\dashrightarrow}
\newcommand{\ti}{\tilde}

\newcommand{\spec}{\operatorname{Spec}}
\newcommand{\proj}{\operatorname{Proj}}

\newcommand{\Supp}{\operatorname{Supp}}

\newcommand{\Exc}{\operatorname{Exc}}
\newcommand{\pr}{\operatorname{pr}}

\newcommand{\vol}{\operatorname{vol}}
\newcommand{\Ivol}{\operatorname{Ivol}}
\newcommand{\codim}{\operatorname{codim}}
\newcommand{\mult}{\operatorname{mult}}

\newcommand{\Aa}{\mathcal{A}}

\newcommand{\Bb}{\mathcal{B}}
\newcommand{\Dd}{\mathcal{D}}
\newcommand{\Ff}{\mathcal{F}}
\newcommand{\Oo}{\mathcal{O}}

\newcommand{\Ii}{\mathcal{I}}
\newcommand{\Jj}{\mathcal{J}}
\newcommand{\Ee}{\mathcal{E}}

\newcommand{\Ll}{\mathcal{L}}
\newcommand{\Tt}{\mathcal{T}}

\newcommand{\Xx}{\mathcal{X}}
\newcommand{\Ss}{\mathcal{S}}
\newcommand{\Yy}{\mathcal{Y}}

\numberwithin{equation}{subsection}

\newtheorem{theorem}{Theorem}[section]
\newtheorem{lemma}[theorem]{Lemma}
\newtheorem{proposition}[theorem]{Proposition}
\newtheorem{definition}[theorem]{Definition}
\newtheorem{example}[theorem]{Example}
\newtheorem{corollary}[theorem]{Corollary}
\newtheorem{remark}[theorem]{Remark}
\newtheorem{conjecture}[theorem]{Conjecture}

\newtheorem*{claim*}{Claim}

 \usepackage{todonotes}
 
\begin{document}

\title[Boundedness of the base varieties of certain fibrations]{Boundedness of the base varieties of certain fibrations}

\subjclass[2020]{14E30, 14D20}

\begin{abstract}
It is conjectured that the base varieties of the Iitaka fibrations are bounded when the Iitaka volumes are bounded above. We confirm this conjecture for  Iitaka $\epsilon$-lc Fano type fibrations.
\end{abstract}

\author{Zhan Li}
\address{Department of Mathematics, Southern University of Science and Technology, 1088 Xueyuan Rd, Shenzhen 518055, China} \email{lizhan@sustech.edu.cn}

\maketitle

\setcounter{tocdepth}{1}
\tableofcontents

\section{Introduction}\label{sec: introduction}

Throughout this paper, we work with varieties defined over complex numbers. 

By analogy with the definition of volumes of divisors, the Iitaka volume of a $\Qq$-Cartier divisor is defined as follows.

\begin{definition}[Iitaka volume]\label{def: Iitaka vol}
Let $X$ be a normal projective variety and $D$ be a $\Qq$-Cartier divisor. When the Iitaka dimension $\kappa(D)$ of $D$ is non-negative, then the Iitaka volume of $D$ is defined to be 
\begin{equation}\label{eq: Iitaka vol}
\Ivol(D)=\Ivol(X, D)\coloneqq \limsup_{m\to\infty} \frac{\kappa(D)!h^0(X, \Oo_X(\lfloor mD \rfloor))}{m^{\kappa(D)}}.
\end{equation} When $\kappa(D)=-\infty$, then we put $\Ivol(D)=0$.
\end{definition}

We study the following problem on the boundedness of base varieties of Iitaka fibrations with fixed (or bounded) Iitaka volumes. For relevant definitions and properties of singularities of pairs, boundedness, Iitaka dimensions/fibrations, DCC/ACC property of sets, etc. see Section \ref{sec: preliminaries}.

\begin{conjecture}\label{conj: bounded of base, klt}
Let $d\in\Nn, v \in \Rr_{>0}$ be fixed numbers, and $\Ii\subset [0,1] \cap \Qq$ be a DCC set. Let $\Ss(d,v, \Ii)$ be the set of varieties $Z$ satisfying the following properties:
\begin{enumerate}
\item $(X, B)$ is klt with $\dim X =d$, and coefficients of $B$ are in $\Ii$,
\item $\Ivol(K_X+B)=v$, and
\item $f: X \dasharrow Z$ is the Iitaka fibration associated with $K_X+B$, where $$Z=\proj \oplus_{m=0}^\infty H^0(X, \Oo_X(\lfloor m(K_X+B) \rfloor)).$$
\end{enumerate}
Then $\Ss(d,v,\Ii)$ is a bounded family.
\end{conjecture}

When $K_X+B$ is big, that is, $\Ivol(K_X+B)=\vol(K_X+B)>0$, Conjecture \ref{conj: bounded of base, klt} is proved by \cite[Theorem 1.1]{HMX18}. For variants of this conjecture, see Section \ref{sec: further discussions}.

\medskip

Our main result is about the boundedness of the base varieties with further assumptions on singularities of $(X, B)$ and the Iitaka fibrations.

\begin{theorem}\label{thm: main result}
Let $d\in\Nn, \epsilon, v \in \Qq_{>0}$ be fixed numbers, and $\Ii\subset [0,1]\cap \Qq$ be a DCC set. Let $\Ff(d,\epsilon, v, \Ii)$ be the set of varieties $Z$ satisfying the following properties:
\begin{enumerate}
\item $(X, B)$ is $\ep$-lc with $\dim X =d$, and coefficients of $B$ are in $\Ii$,
\item $0<\Ivol(K_X+B) < v$,
\item $f: X \dasharrow Z$ is the Iitaka fibration associated with $K_X+B$, where $$Z=\proj \oplus_{m=0}^\infty H^0(X, \Oo_X(\lfloor m(K_X+B) \rfloor)), \text{~and}$$
\item $B$ is big over the generic point $\eta_Z$ of $Z$.
\end{enumerate}
Then $\Ff(d,\ep,v,\Ii)$ is a bounded family.
\end{theorem}

Let $g: X' \to X$ be a birational morphism that resolves $f$ and let $\Exc(g)$ be the sum of reduced exceptional divisors of $g$. If $A$ is a $\Qq$-Cartier ample divisor over $\eta_Z$ such that $0 \leq A \leq B$, then there exists $t>0$ such that $g_*^{-1}B+\Exc(g) \geq tg^*A$ over $\eta_Z$. Thus the assumption in (4) means that $g^{-1}_*B + \Exc(g)$ is big over $\eta_Z$.

Notice that we impose the strong assumption that $B$ is big over $\eta_Z$. This implies that a general fiber of $f$ is birationally bounded (\cite[Theorem 1.3]{HX15}, \cite[Theorem 1.1]{Bir16BAB}). However, it is desirable to obtain the boundedness of the base varieties regardless of the boundedness of fibers. 

\medskip

\cite{Jia18}, \cite{Bir18} and \cite{DCS16} studied similar fibrations (called $(d,r,\ep)$-Fano type fibrations, see Definition \ref{def: (d,r,ep)-Fano type fibration}) where the boundedness of $Z$ is built in the definition. They show the boundedness of the total space under certain assumptions.

\begin{theorem}[{\cite[Theorem 1.3]{Bir18}}]\label{thm: Birkar CY fibration}
Let $d, r$ be natural numbers and $\ep, \delta$ be positive real numbers. Consider the set of all $(d,r,\ep)$-Fano type fibrations $(X,B) \to Z$ and $\Rr$-divisors $0 \leq \Delta \leq B$ whose non-zero coefficients are $\geq \delta$. Then the set of such $(X, \Delta)$ is log bounded.
\end{theorem}

To compare Theorem \ref{thm: main result} with Theorem \ref{thm: Birkar CY fibration}, we note that Theorem \ref{thm: main result} is about the boundedness of the base variety $Z$, while Theorem \ref{thm: Birkar CY fibration} is about the boundedness of the total space $(X, \Delta)$ assuming the boundedness of the base variety $Z$. Combining the two results, we have the following corollary.

\begin{corollary}\label{cor: bounded of the whole family}
Let $d\in\Nn, \epsilon, v \in \Qq_{>0}$ be fixed numbers, and $\Ii\subset [0,1]\cap \Qq$ be a DCC set. Let $\Ss(d, \ep,v,\Ii)$ be the set of log pairs $(X, B)$ satisfying the following properties:
\begin{enumerate}
\item $(X, B)$ is $\ep$-lc with $\dim X =d$, and coefficients of $B$ are in $\Ii$,
\item $K_X+B$ is semi-ample with $\Ivol(K_X+B) < v$,
\item $f: X \to Z$ is the Iitaka fibration associated with $K_X+B$, where $$Z=\proj \oplus_{m=0}^\infty H^0(X, \Oo_X(\lfloor m(K_X+B) \rfloor)), \text{~and}$$
\item $B$ is big over the generic point $\eta_Z$ of $Z$.
\end{enumerate}
Then $\Ss(d,\ep,v,\Ii)$ is a log bounded family.
\end{corollary}

The distribution of Iitaka volumes is closely related to the above boundedness property. By analogy with \cite{Ale94, Kol94, HMX14}, we propose the following conjecture.

\begin{conjecture}\label{conj: DCC}
Let $d\in\Nn$ be a fixed number, and $\Ii\subset [0,1]\cap \Qq$ be a DCC set. Then the set of Iitaka volumes
\[
\{\Ivol(K_X+B) \mid (X, B) \text{~is klt,~} \dim X =d, \text{and coefficients of~} B \text{~are in~} \Ii\}
\] is a DCC set.
\end{conjecture}

When $K_X+B$ is big, this conjecture follows from \cite[Theorem 1.3 (1)]{HMX14} which was originally conjectured by Alexeev and Koll\'ar. We show Conjecture \ref{conj: DCC} under suitable assumptions on the boundedness of the general fibers of the Iitaka fibrations.

\begin{corollary}\label{cor: DCC of Iitaka volume}
Let $d\in\Nn, \epsilon \in \Qq_{>0}$ be fixed numbers, and $\Ii\subset [0,1]\cap \Qq$ be a DCC set. Let $\Dd(d,\epsilon, \Ii)$ be the set of log pairs $(X, B)$ satisfying the following properties:
\begin{enumerate}
\item $(X, B)$ is $\ep$-lc with $\dim X =d$, and coefficients of $B$ are in $\Ii$,
\item $f: X \dasharrow Z$ is the Iitaka fibration associated with $K_X+B$ (if it exists), where $$Z=\proj \oplus_{m=0}^\infty H^0(X, \Oo_X(\lfloor m(K_X+B) \rfloor)), \text{~and}$$
\item $B$ is big over the generic point $\eta_Z$ of $Z$.
\end{enumerate}
Then $\{\Ivol(K_X+B) \mid (X, B) \in \Dd(d,\ep,\Ii)\}$ is a DCC set.
\end{corollary}

We briefly explain the idea for the proof of Theorem \ref{thm: main result}. First, we reduce to the case where $X \to Z$ is a morphism. Then we show that $Z$ is birationally bounded. This means that $Z$ is birational to a fiber of a projective morphism $\mathcal Z' \to \Tt'$ between noetherian schemes of finite type over $\mathbb C$. Using the boundedness of a general fiber of $f$, we show that $(X, B)$ is log birational to a fiber of a family of log pairs which is built upon $\mathcal Z' \to \Tt'$. This is the most technical part of the argument. Finally, the boundedness of $Z$ follows from the finiteness of ample models (\cite[Corollary 1.1.5]{BCHM10}).

\medskip

Another possible approach is to apply the canonical bundle formula to $f: (X, B) \to Z$, then there is a generalized polarized pair $(Z, \Delta_Z+M_Z)$ on $Z$ such that $K_X+B \sim_\Qq f^*(K_Z+\Delta_Z+M_Z)$. One may expect to bound $Z$ under the boundedness of the volumes of $K_Z+\Delta_Z+M_Z$. Such result is only known for surfaces with fixed volumes and some extra conditions (see \cite[Theorem 1.8]{Fil18}). However, it seems that even for $\dim Z=2$, Theorem \ref{thm: main result} does not directly follow from \cite[Theorem 1.8]{Fil18} (the existence of the universal $r\in \Nn$ such that $rM_Z$ is Cartier does not necessarily hold. On the other hand, this is known for a sufficiently high model of $Z$). 

\medskip

The paper is organized as follows. In Section \ref{sec: preliminaries}, we introduce preliminary notation and results. Theorem \ref{thm: main result} and its corollaries are proven in Section \ref{sec: proof of the main result}. In Section \ref{sec: further discussions}, we discuss possible variants of Conjecture \ref{conj: bounded of base, klt}, the relation to the effective adjunction conjecture, and establish some lower dimensional cases. 
\medskip

\noindent\textbf{Acknowledgements}.
The author thanks Chenyang Xu and Lei Zhang for answering related questions. He also thanks Stefano Filipazzi for pointing out some inaccurate statements and providing the reference \cite[Claim 4.38.1]{Kol13}. The author thanks the anonymous referee for reading the manuscript carefully and providing constructive suggestions. This work is partially supported by a starting grant from SUSTech.

\section{Preliminaries}\label{sec: preliminaries}

\subsection{Notation and conventions}\label{subsec: Notation and conventions}
Let $\Ii \subset \Rr$ be a subset, then $\Ii$ is said to be a DCC set (resp. ACC set) if there is no strictly decreasing subsequence (resp. strictly increasing subsequence) in $\Ii$. Let $B$ be an $\Rr$-divisor, then $B \in \Ii$ denotes the fact that the coefficients of $B$ belong to $\Ii$. For a birational morphism $f: Y \to X$ and a divisor $B$ on $X$, $f_*^{-1}(B)$ denotes the strict transform of $B$ on $Y$, and $\Exc(f)$ denotes the sum of reduced exceptional divisors of $f$. A fibration means a projective and surjective morphism with connected fibers. For an $\Rr$-divisor $D$, we denote by $|D|$ the linear system $|\lfloor D \rfloor|$. For $k=\Zz, \Qq, \Rr$, and two divisors $A, B \in k$ on a variety $X$ over $Z$, $A \sim_k B/Z$ means that $A$ and $B$ are $k$-linearly equivalent over $Z$. When $k=\Zz$ or $Z=\spec \Cc$, we omit $k$ or $Z$.

\subsection{Singularities of pairs and boundedness}\label{subsec: singularities of log pairs}

Let $X$ be a normal projective variety and $B$ be an $\Rr$-divisor on $X$, then $(X, B)$ is called a log pair. We assume that $K_X+B$ is an $\Rr$-Cartier divisor for a log pair $(X, B)$. For a divisor $D$ over $X$, if $f: Y \to X$ is a birational morphism from a normal projective variety $Y$ such that $D$ is a divisor on $Y$, then the log discrepancy of $D$ with respect to $(X, B)$ is defined as $\mult_{D}(K_Y-f^*(K_X+B))+1$. This definition is independent of the choice of $Y$. A log pair $(X, B)$ is called sub-klt (resp. sub-lc) if the log discrepancy of any divisor over $X$ is $>0$ (resp. $\geq 0$). If $B \geq 0$, then a sub-klt (resp. sub-lc) pair $(X, B)$ is called klt (resp. lc). For $\ep \in \Rr_{>0}$, a log pair $(X, B)$ is called $\ep$-lc if the log discrepancy of any divisor over $X$ is $>\ep$. For a log pair $(X, B)$ over $Z$, see \cite[Definition 3.6.5]{BCHM10} for the definition of the ample model of $K_X+B$ over $Z$. This is also called the ample model of $(X, B)$ over $Z$. Let $Z \to T$ be a morphism and $B$ be a subscheme of $Z$. If $t\in T$ is a closed point, then $Z_t$ and $B_t$ denote the fiber of $Z$ and $B$ over $t$ respectively.

\begin{definition}[{\cite[Bounded pairs \S 3.5]{HMX14}}]\label{def: bir (log) bounded}\label{def: bounded pairs}
We say that a set $\mathfrak X$ of varieties is birationally bounded if there is a projective morphism $Z \to T$, where $T$ is of finite type, such that for every $X\in \mathfrak X$, there is a closed point $t\in T$ and a birational map $f: Z_t \dasharrow X$.

We say that a set $\mathfrak D$ of log pairs is log birationally bounded (resp. bounded) if there is a log pair $(Z, B)$, where the coefficients of $B$ are all one, and a projective morphism $Z\to T$, where $T$ is of finite type,  such that for every $(X, \Delta)\in \mathfrak D$, there is a closed point $t\in T$ and a birational map $f:Z_t \dasharrow X$ (resp. isomorphism of varieties) such that the support of $B_t$ is not the whole of $Z_t$ and yet $B_t$ contains the support of the strict transform of $\Delta$ and any $f$-exceptional divisor (resp. $f(B_t) = \Supp\Delta$).
\end{definition}

\subsection{Iitaka fibrations and Iitaka volumes}

For this part, we follow the exposition in \cite[\S 2.1]{Laz04I}. Let $X$ be a normal projective variety and $D$ be a $\Qq$-Cartier divisor. Let $${\mathbf N}(X, D) \coloneqq \{m \in \Nn \mid H^0(X, \Oo_X(mD)) \neq 0 \text{~and~} mD \text{~is Cartier}\}.$$ (In particular, ${\mathbf N}(X, D)=\{0\}$ if for any $m\in \Nn_{>0}$ such that $mD$ Cartier, $H^0(X, \Oo_X(mD)) = 0$.) If ${\mathbf N}(X, D)=\{0\}$, then one puts the Iitaka dimension of $D$ to be $-\infty$, and thus $\Ivol(D)=0$. Otherwise, for any $m\in {\mathbf N}(X, D)$, the linear system $|mD|$ defines a rational map $\phi_{|mD|}: X \dasharrow \Pp H^0(X, \Oo_X(mD))$. Then the Iitaka dimension of $D$ is defined to be $$\kappa(D) \coloneqq \max_{m\in {\mathbf N}(X, D)}\{\dim \phi_{|mD|}(X)\}.$$ In this case, there are constants $a, A>0$ such that $$am^{\kappa(D)} \leq h^0(X, \Oo_X(mD)) \leq Am^{\kappa(D)}$$ for all sufficiently large $m\in {\mathbf N}(X, D)$ (\cite[Corollary 2.1.38]{Laz04I}). Hence 
$$
\Ivol(D)=\limsup_{m\to\infty} \frac{\kappa(D)!h^0(X, \Oo_X(\lfloor mD \rfloor))}{m^{\kappa(D)}}
$$ is a positive real number (cf. \cite[Remark 2.1.39]{Laz04I}).

Recall that the volume of $D$ is defined to be 
$$
\vol(D)=\vol(X,D)\coloneqq\limsup_{m\to\infty} \frac{(\dim X)!h^0(X, \Oo_X(\lfloor mD \rfloor))}{m^{\dim X}},
$$ hence $D$ is big iff $\Ivol(D)=\vol(D)>0$.

A particular  case is when $D=K_X+B$ is an adjoint divisor. If $K_X+B$ is semi-ample, then for sufficiently large $m \in {\mathbf N}(X, D)$, $\phi=\phi_{|m(K_X+B)|}: X \to Z$ is a morphism to the ample model of $(X, B)$. Suppose that $K_X+B=\phi^*D_Z$ for some $\Qq$-divisor $D_Z$ on $Z$, then $\Ivol(K_X+B)=\vol(D_Z)=D_Z^{\dim Z}$ by definition. 

\medskip

When $(X, B)$ is a projective klt pair, by \cite[Corollary 1.1.2]{BCHM10}, 
\[
\oplus_{m=0}^\infty H^0(X, \Oo_X(\lfloor m(K_X+B) \rfloor))
\] is finitely generated.  When $\kappa(K_X+B) \geq 0$, then the rational map
\begin{equation}\label{eq: maps to canonical model}
X \dasharrow \proj \oplus_{m=0}^\infty H^0(X, \Oo_X(\lfloor m(K_X+B) \rfloor))
\end{equation} is called the Iitaka fibration associated with $K_X+B$. This is slightly different from \cite[Definition 2.1.34]{Laz04I}, where an Iitaka fibration associated with a Cartier divisor is defined to be a morphism satisfying certain properties. However, the Iitaka fibration in \cite[Definition 2.1.34]{Laz04I} is unique only up to birational equivalence. Hence, in order to make sense of Conjecture \ref{conj: bounded of base, klt}, we call \eqref{eq: maps to canonical model} the Iitaka fibration. 

\subsection{Fano type (log Calabi-Yau) fibrations}

\begin{definition}[$\ep$-lc Fano type fibration]\label{def: ep-lc Fano type fibration}
Let $\ep$ be a positive real number. An $\ep$-lc Fano type (log Calabi-Yau) fibration consists of a pair $(X, B)$ and a fibration $f: X\to Z$ between normal varieties such that we have the following:
\begin{enumerate}
\item $(X,B)$ is a projective $\ep$-lc pair,
\item $K_X + B\sim_\Rr 0/Z$, and
\item $-K_X$ is big over $Z$, i.e. $X$ is of Fano type over $Z$.
\end{enumerate}
\end{definition}

\begin{definition}[$(d,r,\ep)$-Fano type fibration]\label{def: (d,r,ep)-Fano type fibration}
Let $d,r$ be natural numbers and $\ep$ be a positive real number. A $(d,r,\ep)$-Fano type (log Calabi-Yau) fibration consists of a pair $(X, B)$ and a contraction $f: X\to Z$ such that we have the following:
\begin{enumerate}
\item $(X,B)$ is a projective $\ep$-lc pair of dimension $d$,
\item $K_X + B\sim_\Rr f^*L$ for some $\Rr$-divisor $L$,
\item $-K_X$ is big over $Z$, i.e. $X$ is of Fano type over $Z$, 
\item $A$ is a very ample divisor on $Z$ with $A^{\dim Z} \leq r$, and 
\item $A-L$ is ample.
\end{enumerate}
\end{definition}

We emphasize that a very ample or base point free divisor is naturally a Cartier divisor.

\begin{remark}
A $(d,r,\ep)$-Fano type fibration $(X, B) \to Z$ can be viewed as a special case of the fibration map defined in Theorem \ref{thm: main result} given $B \in \Ii$. In fact, for a $(d,r,\ep)$-Fano type fibration $f: X \to Z$ as above, there exists a klt pair $(Z, \Delta)$ such that $K_X+B\sim_\Qq f^*(K_Z+\Delta)$. Then $A - (K_Z+\Delta)$ is ample, and $A^{\dim Z} \leq r$. By length of extremal rays, $K_Z+\Delta+3(\dim Z) A$ is ample. By taking a general $G \in |6(\dim Z)f^*A|$ and replacing $\ep$ by $\min\{\ep, \frac 1 2\}$, then $(X, B+\frac{1}{2}G)$ is $\ep$-lc. Moreover, $K_X+B+\frac{1}{2}G\sim_{\Qq}f^*(K_Z+\Delta+3(\dim Z)A)$ with $K_X+B+\frac{1}{2}G$ semi-ample. Hence $f: X \to Z$ is the Iitaka fibration associated with $K_X+B+\frac{1}{2}G$. By definition, $B$ is big over $Z$. Moreover, 
\[
\begin{split}
&\Ivol(K_X+B+\frac{1}{2}G)=(K_Z+\Delta+3(\dim Z)A)^{\dim Z} \\
\leq& (3(\dim Z)+1)^{\dim Z} A^{\dim Z} \leq (3(\dim Z)+1)^{\dim Z} r.
\end{split}
\] Thus, $(X, B+\frac 1 2 G), Z$ and $v = (3(\dim Z)+1)^{\dim Z} r$ satisfy the assumptions of Theorem \ref{thm: main result} (after possibly enlarging $\Ii$ to $\Ii\cup \{\frac 1 2\}$). 
\end{remark}

\subsection{Canonical bundle formula}\label{subsec: canonical bundle}\label{subsec: canonical bundle formula} We recall the construction of the canonical bundle formula in \cite{Kaw98}. We follow the notions and notation in \cite{FG14} which appear slightly different from \cite{Amb04}. For more about the canonical bundle formula, see \cite{Kaw98, FM00, Amb04, Amb05}, etc. 

The notion of b-divisors is introduced by Shokurov. Let $X$ be a normal variety. An integral b-divisor on $X$ is an element:
\[
\mathbf D \in \mathbf{Div} X \coloneqq \lim_{Y \to X} \mathrm{Div} Y,
\] where the projective limit is taken over all birational models $f: Y \to X$ proper over $X$, under the push-forward homomorphism $f_*: \mathrm{Div} Y \to \mathrm{Div} X$. If $\mathbf D = \sum d_\Gamma \Gamma$ is a b-divisor on $X$, and $Y \to X$ is a birational model of $X$, then the trace of $\mathbf D$ on $Y$ is the divisor 
\[
\mathbf D_Y \coloneqq \sum_{\Gamma \text{~is a divisor on~} Y} d_\Gamma \Gamma.
\] Let $M$ be a Cartier divisor on $X$, then $\overline M$ is the b-divisor such that $(\overline M)_Y=f^*M$ for any birational model $f: Y \to X$. Divisors with coefficients in $\Qq$ or $\Rr$ are defined similarly. For more details, see \cite[\S 2.3]{Cor07}.

Suppose that $(X, B)$ is a log pair with $B$ a $\Qq$-divisor. The discrepancy b-divisor $\mathbf A = \mathbf A(X, B)$ is the $\Qq$-b-divisor of $X$ with the trace $\mathbf A_Y$ defined by the formula $K_Y =f_*(K_X+B)+\mathbf A_Y$, where $f: Y \to X$ is a proper birational morphism of normal varieties. Similarly, we define $\mathbf A^* = \mathbf A^*(X, B)$ by
\[
\mathbf A_{Y}^{*} = \sum_{a_i >-1} a_i A_i
\] for $K_Y =f^*(K_X +B)+\sum a_i A_i$, where $f: Y \to X$ is a proper birational morphism of normal varieties.

\begin{definition}[Klt-trivial and lc-trivial fibrations]\label{def: klt-trivial fibrations}
A klt-trivial (resp. lc-trivial) fibration $f: (X,B) \to Z$ consists of a proper surjective morphism $f: X \to Z$ between normal varieties with connected fibers and a pair $(X,B)$ satisfying the following properties:
\begin{enumerate}
\item $(X, B)$ is sub-klt (resp. sub-lc) over the generic point of $Z$,
\item ${\rm rank} f_*\Oo_X (\lceil \mathbf A(X, B)\rceil) = 1$ (resp. ${\rm rank} f_*\Oo_X (\lceil \mathbf A^*(X, B)\rceil) = 1$), and
\item there exists a $\Qq$-Cartier $\Qq$-divisor $D$ on $Z$ such that 
\[K_X +B\sim_{\Qq} f^*D.\]
\end{enumerate}
\end{definition}

Let $f: (X, B) \to Z$ be an lc-trivial fibration and $P$ be a prime divisor on $Z$. Because $Z$ is normal, after shrinking around $P$, we can assume that $P$ is Cartier. Define
\[
b_P \coloneqq \max\{t \in \Rr \mid (X, B+tf^*P) \text{~is sub-lc over the generic point of~}P\}
\] and set
\[
B_Z \coloneqq \sum_{P} (1-b_P) P, \quad M_Z \coloneqq D-(K_Z+B_Z).
\] Then the following canonical bundle formula holds
\begin{equation}\label{eq: canonical bundle formula}
K_X+B\sim_\Qq f^*(K_Z+B_Z+M_Z).
\end{equation}

In this formula, $B_Z$ is called the divisorial part and $M_Z$ is called the moduli part. When $(X, B)$ is lc, there exist $\Qq$-b-divisors $\mathbf B$ and $\mathbf M$ of $X$ such that $\mathbf B_Z=B_Z$ and $\mathbf M_Z=M_Z$. Moreover, $\mathbf M$ is b-nef and b-abundant in the sense that there is a proper birational morphism $Z' \to Z$ and a proper surjective morphism $h: Z' \to W$ between normal varieties such that (1) $\mathbf M_{Z'} \sim_{\Qq} h^*H$ for some nef and big $\Qq$-divisor $H$ on $W$, and (2) $\mathbf M = \overline{\mathbf M_{Z'}}$. For details, see \cite[Theorem 3.3]{Amb05} and \cite[Theorem 1.1]{FG14}.

The definitions of $B_Z$ and $M_Z$ still make sense when $B$ (in Definition \ref{def: klt-trivial fibrations}) is an $\Rr$-divisor. In this case, we still have b-divisors $\mathbf B$ and $\mathbf M$. When $(X, B)$ is lc over the generic point of $Z$, $M_Z$ is pseudo-effective (\cite[Theorem 3.6]{Bir19}). For more results in this setting, see \cite[\S 3.4]{Bir19} and \cite[\S 6.1]{Bir18}.

\medskip

Next, we have the following effective adjunction conjecture (see \cite[Conjecture 7.13.3]{PS09}). 

\begin{conjecture}[{Effective adjunction}]\label{conj: effective adjunction}
Let $f: (X, B) \to Z$ be an lc-trivial fibration. There exists a positive integer $m$ depending only on the dimension of $X$ and the horizontal multiplicities of $B$ (a finite set of rational numbers) such that $m\mathbf M$ is a base point free b-divisor (i.e. there is a birational morphism between projective varieties $h: Y' \to Y$ such that  $m\mathbf M_{Y'}$ is base point free and $\mathbf M = \overline{\mathbf M_{Y'}}$).
\end{conjecture}

This conjecture is known when general fibers of $f$ are curves (see \cite[Theorem 8.1]{PS09} and references therein).

\section{Proof of Theorem \ref{thm: main result} and its corollaries}\label{sec: proof of the main result}

For convenience, we introduce the following additional notation. Let $D, L$ be two $\Rr$-divisors, then we write $D \preceq_\Rr L$ to indicate that there exists an effective $\Rr$-divisor $E$ such that $D+E \sim_{\Rr} L$. If $D, L$ are $\Qq$-Cartier, then $D \preceq_\Rr L$ implies that $\vol(D) \leq \vol(L)$. We use the notation $m=m(a_1, \ldots, a_k)$ to emphasize that the natural number $m$ depends only on the choice of factors $a_1, \ldots, a_k$. For example, $m=m(\dim X, \ep, \Ii)$ means that $m$ depends only on the dimension of $X$, $\ep$ and the coefficient set $\Ii$.

\medskip

 The proof of Theorem \ref{thm: main result} relies on Proposition \ref{prop: key} whose proof is a bit technical and thus will be given afterwards. We list the maps that shall appear in the proof of Theorem \ref{thm: main result} in the following diagram (starting from Step 2). 

\[
 \xymatrix{
  & X'  \ar[dd]_\pi \ar@/_/[dl]_g \ar@{-->}[rr]^{\theta} \ar@/_/[dddr]^{f'}& &Y \ar@/^/[dddl]^h \ar@/^/[dr]^w&\\
X \ar[dd]_f&&&&Y'\ar@/^/[ddll]^u\\
&\tilde Z \ar[dl]^p \ar[dr]_q&&&\\
Z \ar@{-->}[rr]^\mu&&Z'&&
}
 \]
 
 \begin{proof}[Proof of Theorem \ref{thm: main result}]
 
Let $k=\dim Z$. We can assume that $k>0$. Replacing $\ep$ by $\min\{\ep, \frac 1 2\}$, we can assume that $\ep<1$.

 \medskip
 
 \noindent Step 1. Replace $(X, B)$ by its good minimal model.
 
Take a log resolution $g: X' \to X$ of $(X, B)$ such that $f \circ g: X' \to Z$ is a morphism. Let $K_{X'}+B''=g^*(K_X+B)+E'$, where $B'' = g_*^{-1}B+(1-\ep/2) \Exc(g)$. Then $(X', B'')$ is $\ep/2$-lc and $B''$ is big over $Z$. If $F'$ is a general fiber of $f\circ g$, then $(F', B''|_{F'})$ is klt and $B''|_{F'}$ is big. By \cite[Theorem 1.2]{BCHM10} and the base point free theorem, $(F', B''|_{F'})$ has a good minimal model. Applying \cite[Theorem 4.4]{Lai11}, $(X', B'')$ has a good minimal model. In \cite[Theorem 4.4]{Lai11}, this result is obtained for $X$ with terminal singularities, however, the same argument works for a klt pair. In fact, the terminal singularity assumption is only used to show that when $(X, B)$ is terminal, then a good minimal model of $(X', g_*^{-1}B)$ is also a good minimal model of $(X, B)$ (see \cite[Lemma 2.2]{Lai11}). However, a good minimal model of $(X', B'')$ is still a good minimal model of $(X, B)$ (for example, see \cite[Remark 2.6 (i)]{Bir10}). Thus one can work with $(X', B'')$ in the proof of \cite[Theorem 4.4]{Lai11} and the result follows. 

Replacing $(X, B)$ by a good minimal model of $(X', B'')$, $\ep$ by $\ep/2$, and possibly enlarge $\Ii$ to $\Ii\cup\{1-\ep/2\}$, we can assume that $f: X \to Z$ is a morphism and it is an $\ep$-lc Fano type fibration for $(X, B)$. 

\medskip
 
\noindent Step 2. The construction of $Z'$.
 
By $K_X+B \sim_\Qq 0/Z$ and the canonical bundle formula (see \eqref{eq: canonical bundle formula}), $K_X+B\sim_\Qq f^*(K_Z+\Delta+M)$, where $\Delta$ is the divisorial part and $M$ is the moduli part. As $B\in \Ii$ with $\Ii$ a DCC set, by ACC for log canonical thresholds (see \cite[Theorem 1.1]{HMX14}), there exists a DCC set $\Jj=\Jj(\dim X, \Ii)$ such that $\Delta \in \Jj$. For a general fiber $F_f$ of $f$, $(F_f, B|_{F_f})$ is $\ep$-lc Fano and coefficients of $B|_{F_f}$ are bounded below (in fact, finite by \cite[Theorem D]{HMX14}). Thus $(F_f, B|_{F_f})$ belongs to a bounded family (\cite[Theorem 1.1]{Bir16BAB}, or use \cite[Theorem 1.3]{HX15}). 

We argue as \cite[Theorem 1.4]{HX15} to show that there is an $m=m(d, \Ii, \ep)$ such that $|m(K_Z+\Delta+M)|$ defines a birational map $\mu: Z \dasharrow Z'$. Let $\tau: \bar Z \to Z$ be a log smooth model of $(Z, \Delta)$ such that $\tau^{-1}_*\Delta \cup \Exc(\tau)$ is an snc divisor and the trace $M_{\bar Z}$ of the moduli b-divisor $\mathbf M$ is nef on $\bar Z$. Then by the boundedness of $F_f$, there exists an $r=r(\dim F_f, \ep)\in \Nn$ such that $rM_{\bar Z}$ is an integral divisor (see \cite[Claim 3.2]{HX15} or \cite{Tod10} which both rely on \cite{FM00}). Let $K_{\bar Z}+\bar \Delta + M_{\bar Z} =\tau^*(K_Z+\Delta+M)+E'$ with $\bar \Delta = \tau_*^{-1} \Delta+\Exc(\tau)$. Then $E' \geq 0$ is $\tau$-exceptional, and $({\bar Z}, \bar \Delta + M_{\bar Z})$ is generalized lc (see \cite[Definition 4.1]{BZ16}) with $\bar \Delta \in \Jj \cup\{1\}$. Moreover, $K_{\bar Z}+\bar \Delta + M_{\bar Z}$ is big. By \cite[Theorem 1.3]{BZ16} , there exists an $m=m(d, \Ii, \ep) \in \Nn$, such that $|m(K_{\bar Z}+\bar \Delta + M_{\bar Z})|$ defines a birational map. For an $\Rr$-Cartier divisor $D$ on $Z$, there exists a $\tau$-exceptional divisor $F$ such that $\lfloor \tau^*D \rfloor \leq \tau^{-1}_*\lfloor D \rfloor + F$. Hence if $F'$ is a $\tau$-exceptional divisor such that $|\tau^*D+F'|$ defines a birational map, then $|D|$ also defines a birational map. Apply this to $D=m(K_Z+\Delta+M), F' =mE'$, then $|m(K_Z+\Delta+M)|$ defines a birational map $\mu: Z \dasharrow Z'$.

Take a resolution $p: \tilde Z \to Z$ such that the movable part of $p^*|m(K_Z+\Delta+M)|$ is base point free, and write
\begin{equation}\label{eq: movable is free}
p^*|m(K_Z+\Delta+M)| = |\tilde R|+\tilde F, 
\end{equation} where $|\tilde R|$ is the movable part and $\tilde F$ is the fixed part. This can be done by first passing to a small $\Qq$-factorialization of $Z$ (it exists because there exists $\Delta_Z$ such that $(Z, \Delta_Z)$ is klt by \cite[Theorem 0.2]{Amb05}), and then resolving the $\Qq$-Cartier divisor which is the strict transform of $\lfloor m(K_Z+\Delta+M)\rfloor$ on this small $\Qq$-factorial model. Now, $|\tilde R|$ defines a morphism $q: \tilde Z \to Z'$ which resolves $\mu\circ p$. Let $A$ be the very ample divisor on $Z'$ such that $\ti R=q^*A$. Then
\begin{equation}\label{eq: bound A^k}
A^k \leq \vol(\tilde Z, \tilde R+\tilde F) \leq \vol(Z, m(K_Z+\Delta+M)) \leq m^k v.
\end{equation} Thus $Z'$ belongs to a bounded family.

\medskip

 \noindent Step 3. The construction of $(X', D')$.

Let $g: X' \to X$ be a log resolution of $(X, B)$ which also resolves $X \dto \tilde Z$. Let $\pi: X' \to \tilde Z$ be the corresponding morphism and $f'=q \circ \pi$. Set 
\begin{equation}\label{eq: D'}
D' = (1+\frac \ep 2) g_*^{-1}B + (1-\frac \ep 2)\Exc(g).
\end{equation} As $(X, B)$ is $\ep$-lc, $(X', D')$ is $\frac \ep 2$-lc. For $0<\delta \ll 1$, 
\begin{equation}
K_{X'}+D' \geq g^*(K_X+(1+\delta)B),
\end{equation} thus $K_{X'}+D'$ is big$/Z$. $K_{X'}+D'$ is also big$/Z'$ as $Z \dto Z'$ is birational. By \cite[Theorem 1.2]{BCHM10}, there exists a minimal model $\theta: X' \dto Y/Z'$ of $(X', D')/Z'$. Let $D_Y = \theta_*D'$, then $K_Y+D_Y$ is semi-ample$/Z'$ by the base point free theorem.

\medskip

 \noindent Step 4. Log birational boundedness of $(X, B)$.

We claim that $K_Y+D_Y+3dh^*A$ is a nef and big divisor. Recall that $\dim Y=d$ and $A$ is very ample on $Z'$ such that $\ti R=q^*A$. If there is an extremal curve $C$ such that $C \cdot (K_Y+D_Y+3dh^*A)<0$, then $h_*C\neq 0$ because $K_Y+D_Y$ is nef$/Z'$. Moreover, $C \cdot (K_Y+D_Y)<0$. By the length of extremal rays, there exists a curve $C'$ on the same ray class of $C$ such that $-(2d+1) \leq C' \cdot (K_Y+D_Y)<0$. But $C'\cdot(3dh^*A) \geq 3d$, we get $C' \cdot (K_Y+D_Y+3dh^*A)>0$, a contradiction. Next, $K_X+B$ is pseudo-effective and thus $K_{X'}+D'$ is pseudo-effective. As the image of $K_{X'}+D'$, $K_Y+D_Y$ is also pseudo-effective. Let $(Y', D_{Y'})$ be the ample model of $(Y, D_Y)$ over $Z'$ with $w: Y \to Y'$ and $u: Y' \to Z'$ the corresponding morphisms. Then $K_{Y'}+D_{Y'}+Nu^*A$ is ample for $N \gg 1$. Thus 
\begin{equation}\label{eq: bigness}
K_Y+D_Y+3dh^*A = (1-\frac 1 N)(K_Y+D_Y)+(\frac 1 N (K_Y+D_Y) + 3dh^*A)
\end{equation} is big.

Let $G \in |6dh^*A|$ be a general element, then $(Y, D_Y+\frac 1 2 G)$ is still $\frac \ep 2$-lc with $D_Y+\frac 1 2 G \in \Ii'$, where $\Ii' =(1+\frac \ep 2)\Ii \cup\{1-\frac \ep 2\}\cup\{\frac 1 2\}$ is a DCC set. By \cite[Theorem 1.3 (3)]{HMX14}, there exists $m'=m'(\dim Y, \Ii') \in \Nn$ such that the linear system $|m'(K_Y+ D_Y+\frac 1 2 G)|$ defines a birational map $\varphi: Y \dasharrow Y''$.   By Proposition \ref{prop: key}, 
\begin{equation}\label{eq: N}
\vol(Y, (K_Y+ D_Y+\frac 1 2 G))\leq N(\Ii, \ep, d,v),
\end{equation} where $N(\Ii, \ep, d,v)$ depends only on $\Ii, \ep, d,v$. Moreover, by \cite[Lemma 7.3]{HMX14}, there is a rational number $\beta=\beta(d,\Ii')<1$ such that $K_Y+\beta(D_Y+\frac 1 2 G)$ is still big.

By \cite[Lemma 2.4.2 (3)]{HMX13}, to show that the family $\{(X, B)\}$ is log birationally bounded, it is enough to show that $\{(Y, D_Y)\}$ is log birationally bounded. Then by \cite[Lemma 2.4.2 (4)]{HMX13}, it is enough to show that 
\[
(\Supp D_{Y''} + \Exc(\varphi^{-1})) \cdot (H)^{d-1}
\] is bounded above, where $D_{Y''}$ is the strict transform of $D_Y$ and $H$ is the very ample divisor on $Y''$ determined by $\varphi$. As $\Ii'$ is a DCC set, let $0<\delta = \min \Ii'$. Let $\pi_1: W \to Y, \pi_2: W \to Y''$ be a common resolution such that $\varphi \circ \pi_1 = \pi_2$ and the movable part of $\pi_1^*|m'(K_Y+D_Y+\frac 1 2 G)|$ is base point free. Let $D_W=\pi_{1*}^{-1}D_Y$. By \cite[Lemma 3.2]{HMX13},
\begin{equation}\label{eq: estimate1}
\begin{split}
&(\Supp D_{Y''} + \Exc(\varphi^{-1})) \cdot (H)^{d-1}\\
\leq &(\Supp D_{W} + \Exc(\pi_1)) \cdot (\pi_2^*H)^{d-1}\\
\leq & {2^d}\vol(W, K_{W}+\Supp D_{W} + \Exc(\pi_1) +2(2d+1) \pi_2^*H)\\
\leq & {2^d}\vol(W, K_{W}+\frac 1 \delta D_{W} + \Exc(\pi_1) +2(2d+1) \pi_2^*H)\\
\leq & {2^d}\vol(W, \pi_1^*(K_{Y}+\frac 1 \delta D_{Y}) + E \\
&\qquad +2(2d+1) \pi_1^*(m'(K_Y+D_Y+\frac 1 2 G))),
\end{split}
\end{equation} where $E$ is a $\pi_1$-exceptional divisor with sufficiently large coefficients. Since $K_Y+\beta(D_Y+\frac 1 2 G)$ is big, there is $\Theta \geq 0$ such that $\Theta+(1-\beta)D_Y \sim_\Rr K_Y+D_Y+\frac 1 2 G$. Hence $D_Y \preceq_\Rr \frac{1}{1-\beta}(K_Y+D_Y+\frac 1 2 G)$. In particular, \[
K_Y+\frac 1 \delta D_Y \preceq_{\Rr} (1+\frac{1}{\delta(1-\beta)})(K_Y+D_Y+\frac 1 2 G)).
\] Now, continue the estimates in \eqref{eq: estimate1}, we have
\begin{equation}\label{eq: estimate2}
\begin{split}
&(\Supp D_{Y''} + \Exc(\varphi^{-1})) \cdot (H)^{d-1}\\
\leq & {2^d}\vol(Y, (K_{Y}+\frac 1 \delta D_{Y}) + 2(2d+1)(m'(K_Y+D_Y+\frac 1 2 G)))\\
\leq & {2^d}\vol((1+\frac{1}{\delta(1-\beta)}+2(2d+1)m')(K_Y+D_Y+\frac 1 2 G))\\
\leq & 2^d(1+\frac{1}{\delta(1-\beta)}+2(2d+1)m')^d N(\Ii, \ep, d,v).
\end{split}
\end{equation} This final estimate depends only on $\Ii, \ep, d,v$. Hence the set of log pairs $\{(X, B)\}$ in Theorem \ref{thm: main result} is log birationally bounded.

\medskip

 \noindent Step 5. Boundedness of $Z$.
  
The argument below follows the same lines as \cite[Theorem 1.6]{HMX14}. However, \cite[Theorem 1.6]{HMX14} deals with the $K_X+B$ ample case (see \cite[Lemma 9.1]{HMX14}) while $K_X+B$ is only semi-ample here. Hence, we provide details on this part. 

First, by Step 4, there exists a projective morphism $\Xx \to \mathcal T$ between schemes of finite type, and a reduced subscheme $\Bb \subset \Xx$ such that  for any $(X, B)$ in Theorem \ref{thm: main result}, there exists $t\in \Tt$ and a birational map $i: X \dto \Xx_t$ satisfying $\Supp i_{*}^{-1}(B) \cup \Exc(i^{-1}) \subset \Supp\Bb_t$. Taking log resolutions and stratifying $\Tt$ into a disjoint union of subvarieties, we can assume that $(\Xx, \Bb)$ is log smooth over $\Tt$ (i.e. $\Xx$ as well as any stratum of $\Bb$ are smooth over $\Tt$). Passing to a finite cover of $\Tt$ (see \cite[Claim 4.38.1]{Kol13}), we can assume that every stratum of $\Bb$ has irreducible fibers over $\Tt$. For $(\Xx, (1-\ep)\Bb)$, by a sequence of blowups of strata, we extract all the divisors whose log discrepancies are $\leq 1$. Then one can show that $i^{-1}: \Xx_t \dasharrow X$ is a birational contraction, that is, $i$ does not contract any divisor on $X$. This step is done in \cite[Proof of (1.6), Page 556--557]{HMX14}.

Next, let $0<\delta\coloneqq \min \Ii$. We define a $\Qq$-divisor $\bar \Bb$ on $\Xx$ corresponding to $(X, B)$ as follows. Let $i^{-1}: \Xx_t \dasharrow X$ be the above birational contraction. Then $\bar\Bb$ is the $\Qq$-divisor such that $\Supp \bar\Bb \subset \Supp \Bb$ and $\bar\Bb_t$ is the sum of the strict transform of $B$ and $(1-\ep)\Exc(i^{-1})$. This $\bar\Bb$ is uniquely determined because $(\Xx, \Bb)$ is log smooth over $\Tt$ and each stratum of $\Bb$ has irreducible fibers over $\Tt$. We use \cite[Corollary 1.1.5 (2)]{BCHM10} to show the following claim:

Suppose that $(\Xx_t, \bar\Bb_t)$ corresponds to some $(X, B)$. Then there are finitely many ample models over $\Tt$ for 
\[
\{ (\Xx, \Bb') \mid \Supp\Bb'=\Supp \bar\Bb, \text{~and~} \delta \Supp \bar\Bb \leq \Bb' \leq (1-\ep)\Supp \bar\Bb\}.
\]

We can choose a rational number $0 < \tau \ll 1$ such that $K_X+(1+\tau)B$ is big by the same argument as \eqref{eq: bigness}. Then $K_{\Xx}+(1+\tau)\bar\Bb$ is big over $\Tt$ by \cite[Theorem 1.2]{HMX18}. Let $\Aa+\Ee \sim_\Qq K_{\Xx}+(1+\tau)\bar\Bb/\Tt$ with $\Ee \geq 0$ and $\Aa>0$ an ample $\Qq$-Cartier divisor which has no components in common with $\Ee$ and $\Supp \bar\Bb$. Choose $\alpha \in \Qq_{>0}$ such that $(\Xx, \alpha(\Aa+\Ee)+(1-\ep/2) \Supp \bar\Bb)$ is klt. By choosing $\gamma \in \Qq_{>0}$ sufficiently small, we can assume that $\Delta'$ in 
\[
(1+\gamma\alpha)(K_\Xx+\Bb') = K_\Xx+\gamma\alpha(K_\Xx+(1+\tau)\bar\Bb)+\Delta'
\] is effective with coefficients $<(1-\ep/2)$. Indeed, $\Delta' = (1+\gamma\alpha)\Bb'-\gamma\alpha(1+\tau)\bar\Bb$, where $\Bb'$ has coefficients in $[\delta, (1-\ep)]$ and $\Supp \Bb' = \Supp \bar\Bb$. Moreover, $\Supp \Delta' \subset \Supp \bar\Bb$. Now $(\Xx, \gamma\alpha \Aa+\Delta'+\gamma\alpha \Ee)$ is klt with
\[
(1+\gamma\alpha)(K_\Xx+\Bb') \sim_\Qq K_\Xx+\gamma\alpha \Aa+\Delta'+\gamma\alpha \Ee/\Tt.
\] Then the ample model of $(\Xx, \Bb')/\Tt$ is also the ample model of $(\Xx, \gamma\alpha \Aa+\Delta'+\gamma\alpha \Ee)/\Tt$. By \cite[Corollary 1.1.5 (2)]{BCHM10}, we have finitely many rational maps $\psi_j: \Xx \dasharrow \mathcal Z_j$ over $\Tt$ such that for a klt pair $(\Xx, \gamma\alpha \Aa+\Delta'+\gamma\alpha \Ee)$ with $\Supp \Delta' \subset \Supp \bar\Bb$ and $K_\Xx+\gamma\alpha \Aa+\Delta'+\gamma\alpha \Ee$ pseudo-effective, there exists a $\psi_j$ which is the ample model of $(\Xx, \gamma\alpha \Aa+\Delta'+\gamma\alpha \Ee)$ over $\Tt$. Hence, $\psi_j$ is also the ample model of $(\Xx, \Bb')$ over $\Tt$.

Finally, by \cite[Corollary 1.4]{HMX18}, if $\psi_j: \Xx \dasharrow \mathcal Z_j/\Tt$ is the ample model of $(\Xx, \Bb')$, then $\psi_{j,t}: \Xx_t \dasharrow \mathcal Z_{j,t}$ is the ample model of $(\Xx_t, \Bb'_t)$ for every closed point $t\in\Tt$. Because $\Xx_t \dasharrow X$ is a birational contraction, the ample model of $(X, B)$ is the ample model of $(\Xx_t, \bar\Bb_t)$. In fact, let $r_1: W \to X$, $r_2: W \to \Xx_t$ be a common log resolution that also resolves $i: X \dasharrow \Xx_t$. Then $K_W+r_{1,*}^{-1}B +(1-\ep)\Exc(r_1) = r_1^*(K_X+B)+E_1$ and $K_W+r_{2,*}^{-1}\bar\Bb_t +(1-\ep)\Exc(r_2) = r_1^*(K_{\Xx_t}+\bar\Bb_t)+E_2$ with $E_1, E_2 \geq 0$. But $r_{1,*}^{-1}B +(1-\ep)\Exc(r_1)= r_{2,*}^{-1}\bar\Bb_t +(1-\ep)\Exc(r_2)$ by the construction of $\bar\Bb$. Then the assertion follows because the ample model of $(W, r_{1,*}^{-1}B +(1-\ep)\Exc(r_1))$ (resp. $(W, r_{2,*}^{-1}\bar\Bb_t +(1-\ep)\Exc(r_2))$) is the same as the ample model of $(X, B)$ (resp. $(\Xx_t, \bar\Bb_t)$). Because $K_X+ B$ is semi-ample, $Z=\proj\oplus_{m=0}^\infty H^0(X, \Oo_X(\lfloor m(K_X+B)\rfloor))$ is exactly the ample model of $(X, B)$ (\cite[Lemma 3.6.6 (3)]{BCHM10}), this shows the boundedness for $\Ff(d,\ep,v,\Ii)$.
\end{proof}

Next, we show Proposition \ref{prop: key} which is used in Step 4 in the proof of Theorem \ref{thm: main result}. The strategy is similar to \cite[Theorem 4.1]{Jia18} (also see \cite[Proposition 4.1]{DCS16} and \cite[Proposition 5.8]{Bir18}). However, it is much more subtle here because $h: Y \to Z'$ in the proof of Theorem \ref{thm: main result} may not be an $\ep$-lc Fano type fibration. Under the notation in the proof of Theorem \ref{thm: main result}, we have the following auxiliary constructions: 

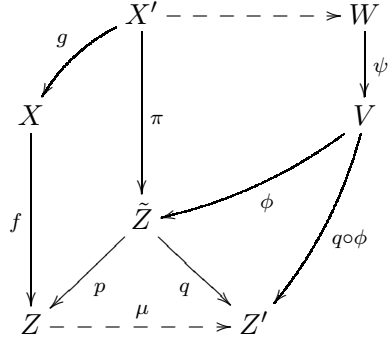
\[
 \xymatrix{
  & X' \ar[dd]^\pi \ar@/_/[dl]_g \ar@{-->}[rr] & &W  \ar[d]^\psi\\
X \ar[dd]_f&&&V \ar@/^/[dll]^\phi \ar@/^/[ddl]^{q\circ\phi}\\
&\tilde Z \ar[dl]^p \ar[dr]_q& &\\
Z \ar@{-->}[rr]^\mu&&Z'&
}
 \]
 
Recall that $g: X' \to X$ is a log resolution of $(X, B)$. Let $B'= g_*^{-1}B + (1-\ep/2) \Exc(g)$, then
\[
K_{X'}+B' =g^*(K_X+B)+E',
\] where $E'>0$ is $g$-exceptional. There is a non-empty smooth open set $U_{\ti Z}\subset \tilde Z$ such that $(X'|_{\pi^{-1}(U_{\ti Z})}, B'|_{\pi^{-1}(U_{\ti Z})})$ is log smooth over $U_{\ti Z}$. Notice that $p: \tilde Z \to Z$ is birational. For a general closed point $t \in U_{\ti Z}$, let $F_\pi$ be the fiber $\pi^{-1}(t)$ and $F_f$ be the fiber $f^{-1}(p(t))$. Then
\[
K_{F_{\pi}}+B'|_{F_\pi} =(g|_{F_\pi})^*(K_{F_f}+B|_{F_f})+E'|_{F_\pi},
\] where $E'|_{F_\pi} \geq 0$ is an exceptional divisor for $g|_{F_\pi}: F_\pi \to F_f$. Because $K_{F_f}+B|_{F_f} \sim_\Qq 0$, $(F_{\pi}, B'|_{F_\pi})$ has a good minimal model. In fact, we can run a partial $(K_{F_{\pi}}+B'|_{F_\pi})$-MMP over $F_f$. After finitely many steps, $E'|_{F_\pi}$ must be contracted and thus the resulting log pair is crepant over $(F_f, B|_{F_f})$. This is a good minimal model for $(F_{\pi}, B'|_{F_\pi})$ (alternatively, this follows from \cite[Theorem 1.2]{BCHM10} and the base point free theorem as before). Then by \cite[Theorem 1.2]{HMX18}, $(X'|_{\pi^{-1}(U_{\ti Z})}, B'|_{\pi^{-1}(U_{\ti Z})})$ has a good minimal model over $U_{\ti Z}$. By \cite[Theorem 1.1]{HX13} (also see \cite[Theorem 1.4]{Bir12b}), $(X', B')$ has a good minimal model $(W, B_W)$ over $\ti Z$. Let $\psi: W \to V/\ti Z$ be the morphism associated with $K_W+B_W$. Then $K_W+B_W \sim_\Qq 0/V$. Let $\phi: V \to \ti Z$ be the corresponding birational morphism.

Next, let 
\begin{equation}\label{eq: def of bar B'}
K_{X'}+\bar B' = g^*(K_X+B),
\end{equation} then $(X', \bar B')$ is a subklt pair such that $K_{X'}+\bar B' \sim_\Qq \pi^*p^* L$, where $$L\coloneqq K_Z+\Delta+M$$ (see Step 2 in the proof of Theorem \ref{thm: main result}). In particular, $K_{X'}+\bar B'\sim_\Qq 0/\ti Z$. Let $K_W+\bar B_W$ be the pushforward of $K_{X'}+\bar B'$ on $W$.
Then $(W, \bar B_W)$ is subklt as $K_{X'}+\bar B' \sim_\Qq 0/\ti Z$. By the canonical bundle formula, 
\begin{equation}\label{eq: canonical bundle for W, bar B}
K_W+\bar B_W \sim_\Qq \psi^*(K_V+\bar \Delta_V+M_V),
\end{equation} where $\bar \Delta_V$ is the divisorial part and $M_V=L_V - (K_V+\bar \Delta_V)$ is the moduli part with $L_V \coloneqq \phi^*p^*L$. Apply the canonical bundle formula to the klt-trivial fibration $\psi: (W, B_W) \to V$, we have
\begin{equation}\label{eq: canonical bundle for W, B}
K_W+B_W \sim_\Qq \psi^*(K_V+\Delta_V+M_V).
\end{equation} Notice that the moduli parts in \eqref{eq: canonical bundle for W, bar B} and \eqref{eq: canonical bundle for W, B} can be chosen as the same $\Qq$-divisor by the following reason (also see \cite[Lemma 3.5]{Bir19}).  

By construction, $K_W+B_W=K_W+\bar B_W+E_W$ for some $E_W \geq 0$. Because $K_W+B_W \sim_\Qq 0 / V$ and $K_W+\bar B_W+E_W\sim_\Qq 0 / V$, we have $E_W\sim_\Qq 0/V$. Then $E_W = \psi^*E_V$ for some $\Qq$-divisor $E_V\geq 0$ on $V$. This can be derived from \cite[Lemma 1.7]{KMM94}. Although $V$ may not be $\Qq$-factorial here, it is enough to show our claim on the smooth locus of $V$ because we have $E_W \sim_\Qq 0/V$ instead of merely $E_W \equiv 0/V$.

\iffalse
This fact should be well known. Indeed, because $V$ is normal, it is enough to show the claim on the smooth locus of $V$. Thus, we can assume that $V$ is smooth. By $E_W \sim_\Qq 0/V$ and $E_W \geq 0$, $E_W$ is vertical over $V$. For any divisor $P$ on $V$, let 
$$t_P = \max\{t \in \Qq_{\geq 0}\mid E_W-t\psi^* P \geq 0 \text{~over the generic point~}\eta_P\}.$$ Set $E_V= \sum t_P P \geq 0$. Over the generic point of each divisor of $V$, $E_W-\psi^*E_V$ is very exceptional  (see \cite[Definition 3.2]{Sho03}) and $E_W-\psi^*E_V \sim_\Qq 0$. Hence by the negativity lemma (for example, see \cite[Lemma 3.3]{Bir12b}), $E_W-\psi^*E_{V}=0$ over the generic point of each divisor of $V$. Thus, write $E_W-\psi^*E_{V}=F^+ -F^-$ where $F^+, F^- \geq 0$ are divisors without common components, and $\codim_{V} \psi(F^+)\geq 2, \codim_{V} \psi(F^-) \geq 2$. By $E_W-\psi^*E_{V} \sim_\Qq 0/V$, $F^+ -F^- \sim_\Qq 0/V$. Apply the negativity lemma again, we have $F^+=F^-=0$. Thus $E_W=\psi^*E_{V}$. 
\fi

By the definition of the divisorial parts, $\Delta_V = \bar \Delta_V+E_V$. Notice that 
\[
K_W+B_W \sim_\Qq \psi^*(L_V+E_V).
\] Hence the moduli part for $(W, B_W) \to V$ can be chosen as
\[
L_V+E_V -(K_V+\Delta_V),
\] which is exactly $L_V- (K_V+\bar \Delta_V)=M_V$. 

Recall that in Step 2 of the proof of Theorem \ref{thm: main result}, we have $K_X+B\sim_\Qq f^*(K_Z+\Delta+M)$. By above discussion, for $\eta \coloneqq p \circ\phi $,
\begin{equation}\label{eq: eta}
K_V+\Delta_V+M_V= \eta^*(K_Z+\Delta+M)+E_V.
\end{equation}

Next, recall that $p: \ti Z \to Z$ is a log resolution such that $p^*|m(K_Z+\Delta+M)|=|\ti R| +\ti F$ (see \eqref{eq: movable is free}), where $|\ti R|$ is base point free and defines a birational morphism $q: \ti Z \to Z'$. Moreover, $Z'$ belongs to a bounded family which depends only on $\Ii, \ep, d, v$ (see Step 2 in the proof of Theorem \ref{thm: main result}, notice that we do not use Proposition \ref{prop: key} until Step 4). Let $R_V =\phi^*\ti R$ and $\dim V =\dim Z = k$. The key estimate is the following lemma.

\begin{lemma}\label{le: key}
Under the assumptions in Theorem \ref{thm: main result}, using the notation introduced above, there exists $\mathfrak R = \mathfrak R (\Ii, \ep, d, v) \in \Qq_{>0}$ such that $(K_V+\Delta_V+M_V) \cdot { R_V}^{k-1} < \mathfrak R$.
\end{lemma}
\begin{proof}
We estimate $K_V \cdot { R_V}^{k-1} , \Delta_V \cdot { R_V}^{k-1}$ and $M_V \cdot { R_V}^{k-1} $ separately.

\medskip

First, we work with $K_V \cdot { R_V}^{k-1}$. Recall that $\ti R$ defines $q: \ti Z \to Z'$, and for the very ample divisor $A$ on $Z'$ such that $\ti R=q^*A$, we have $A^k \leq m^k v$ (see \eqref{eq: bound A^k}). Hence $Z' \subset \Pp^l$ with $l \leq m^kv+k-1$ (for example, see \cite[Proposition 0]{EH87}), and $\Oo_{Z'}(A) \simeq \Oo_{Z'}(1)$. Thus the degree of $Z'$ is bounded above by $m^kv$. Then by the construction of Chow varieties (for example, see \cite[Chapter I.3]{Kol96}), there is a projective morphism between noetherian schemes of finite type $\mathcal U \to \mathcal C$, and a line bundle $\Aa$ on $\mathcal U$ such that for any $Z'$, there is some $t_0$ whose fiber $\mathcal U_{t_0} \simeq Z'$ and $\Aa|_{\mathcal U_{t_0} } \simeq \Oo_{Z'}(A)$. In fact, $\Aa$ is the restriction of ${\rm pr_2}^*\Oo_{\Pp^l}(1)$ to the universal family $\mathcal U \subset \mathcal C \times \Pp^l$. In particular, $K_{Z'} \cdot A^{k-1}$ can only achieve finitely many values and thus be bounded above. By projection formula (\cite[Chapter VI, Proposition 2.11]{Kol96}), $K_V \cdot R_V^{k-1} = K_{Z'} \cdot A^{k-1}$ which is also bounded above depending only on $\Ii, \ep, d, v$.

\medskip

Next, we work with $\Delta_V \cdot { R_V}^{k-1}$. Notice that $B_W \in \Ii \cup \{1-\frac \ep 2\}$ which is a DCC set. By ACC for log canonical thresholds \cite[Theorem 1.1]{HMX14}, $\Delta_V$ belongs to a DCC set depending on $\Ii, \ep, d$. In particular, each coefficient of $\Delta_V$ is greater than some $\delta=\delta(\Ii, \ep,d)\in \Qq_{>0}$. By \cite[Theorem 8.1]{BZ16} (this is \cite[Lemma 7.3]{HMX14} for the generalized polarized pairs), and use the same argument as Step 2 in the proof of Theorem \ref{thm: main result} (i.e. take a smooth model $\bar V \to V$ such that the moduli part is a nef $\Qq$-Cartier divisor, and use the boundedness of general fibers to bound the Betti numbers), there exists $e=e(\ep, \Ii, d)\in (0,1) \cap \Qq$ such that $K_V+e\Delta_V+eM_V$ is still big. Because $M_V$ is pseudo-effective,
\[
(1-e)\Delta_V \preceq_{\Rr} (1-e)\Delta_V+(K_V+e\Delta_V+M_V) = K_V+\Delta_V+M_V.
\] By \cite[Lemma 3.2]{HMX13}, if $H=2(2k+1)R_V$, then
\[
\Supp \Delta_V \cdot H^{k-1} \leq 2^k \vol(V, K_V+\Supp \Delta_V+H).
\] Hence, it is enough to show that $\vol(V, K_V+\Supp \Delta_V+H)$ has an upper bound which depends only on $\Ii, \ep, d, v$. Because coefficients of $\Delta_V$ are greater than $\delta$, 
\begin{equation}\label{eq: bound Delta}
\Supp \Delta_V \leq \frac 1 \delta \Delta_V \preceq_\Rr \frac{1}{\delta(1-e)} (K_V+\Delta_V+M_V).
\end{equation} By construction, $R_V \preceq_\Rr mL_V$, hence $H=2(2k+1) R_V \preceq_{\Rr} 2(2k+1)m L_V$.  Then by \eqref{eq: bound Delta},
\[
\begin{split}
&\vol(V, K_V+\Supp \Delta_V+H) \\
 \leq & \vol(V, K_V+\frac{1}{\delta(1-e)} (K_V+\Delta_V+M_V)+H)\\
 \leq & \vol(V, K_V+\Delta_V+M_V+\frac{1}{\delta(1-e)} (K_V+\Delta_V+M_V)+2(2k+1)m L_V)\\
  \leq & \vol(V, (1+\frac{1}{\delta(1-e)})(L_V+E_V)+2(2k+1)m L_V)\\
   \leq &  \vol(V, (1+\frac{1}{\delta(1-e)}+2(2k+1)m)(L_V+E_V)), 
\end{split}
\] where the third inequality uses \eqref{eq: eta}. By construction, 
\[
\begin{split}
 \vol(V, L_V+E_V) &=  \vol(V, K_V+\Delta_V+M_V) =\Ivol(W, K_W+B_W) \\
 &= \Ivol(X', K_{X'}+B')= \Ivol(X, K_{X}+B) \leq v,
 \end{split}
\] and thus $\vol(V, K_V+\Supp \Delta_V+H)$ has an upper bound which depends only on $\Ii, \ep, d, v$.

\medskip

Finally, we bound $M_V \cdot { R_V}^{k-1}$. By the above discussion, $M_V$ is the moduli part for both $(W, \bar B_W) \to V$ and $(W, B_W) \to V$, hence $\eta_*M_V = M$. Because $(X, B) \to Z$ is a klt-trivial fibration, there is a klt pair $(Z, \Delta_Z)$ such that $K_X+B \sim_\Qq f^*(K_Z+\Delta_Z)$ (\cite[Theorem 0.2]{Amb05}). Thus there is a small $\Qq$-factorialization $Z^q \to Z$. Moreover, $p$ can be assumed to factor through $Z^q$, and let $p': \ti Z \to Z^q$ be the corresponding morphism. Let $\Delta^q, M^q$ be the strict transforms of $\Delta, M$. We claim that $(p'\circ\phi)^*M^q = M_V+F_V$, where $F_V$ is effective. In fact, because the moduli b-divisor $\mathbf M$ is b-nef, there is a birational model $\bar\phi: \bar V \to V$ such that the trace $\mathbf M_{\bar V}$ is nef. By the negativity lemma, there exists an effective divisor $F_{\bar V}$ such that $\mathbf M_{\bar V}+F_{\bar V} = (p'\circ\phi\circ\bar\phi)^*M^q$. Thus $M_V + \bar\phi_*F_{\bar V} = \bar\phi_*(\mathbf M_{\bar V}+F_{\bar V})=(p'\circ\phi)^*M^q$, where $F_V = \bar\phi_*F_{\bar V}  \geq 0$. By the same argument as before, there exists $e'=e'(\ep, \Ii, d)\in (0,1) \cap \Qq$ such that $K_{Z^q}+e'\Delta^q+e'M^q$ is big. Hence
\begin{equation}\label{eq: bound M}
\begin{split}
&(1-e')M_V \preceq_\Rr (1-e')(p'\circ\phi)^*M^q \\
\preceq_\Rr &(p'\circ\phi)^*(K_{Z^q}+\Delta^q+e'M^q+(1-e')M^q)=L_V.
\end{split}
\end{equation} Because $L_V, R_V$ are nef, and $R_V \preceq_\Rr mL_V$, we have
\[
L_V\cdot R_V^{k-1} \leq L_V\cdot (mL_V)\cdot R_V^{k-2} \leq \cdots \leq L_V \cdot (mL_V)^{k-1} \leq m^{k-1}v.
\] Thus by \eqref{eq: bound M},
\[
M_V \cdot R_V^{k-1} \leq \frac{m^{k-1}v}{(1-e')}.
\]

Put the above estimates together and notice that $m=m(d, \Ii, \ep)$ (see Step 2 in the proof of Theorem \ref{thm: main result}), we have $\mathfrak R = \mathfrak R (\Ii, \ep, d, v) \in \Qq_{>0}$ such that $(K_V+\Delta_V+M_V) \cdot { R_V}^{k-1} < \mathfrak R$.
\end{proof}

\begin{remark}
(1) One cannot apply the method for estimating $M_V \cdot R_V^{k-1}$ to $B_V \cdot R_V^{k-1}$. The reason is that we do not have a bounded $\zeta = \zeta(\Ii, \ep, d, v) \in \Rr_{>0}$ such that $B_V \leq \zeta (p'\circ\phi)^*B^q$.

(2) One can show that $E_V$ in \eqref{eq: eta} is $\eta$-exceptional. Hence
\[
(K_V+\Delta_V+M_V)\cdot L_V^{k-1} = (K_Z+\Delta+M)\cdot L^{k-1} \leq v.
\] However, such estimate breaks in the induction argument of Proposition \ref{prop: key} (see \eqref{eq: induction bound}). Hence, we have to use the above more complicated estimate. 
\end{remark}

In the proof of Proposition \ref{prop: key}, we use the following lemma.

\begin{lemma}[{\cite[Lemma 2.5]{Jia18}}]\label{le: jiang}
Let $X$ be a projective normal variety and let $D$ be an $\Rr$-Cartier $\Rr$-divisor on $X$. Let $S$ be a base point free Cartier prime divisor on $X$. Then for any rational number $q > 0$,
\[
\vol(X,D+qS) \leq \vol(X,D)+q \cdot \dim X \cdot \vol(S,D|_S +qS|_S).
\]
\end{lemma}

\begin{lemma}\label{le: for key prop}
Let $d\in \Nn$, $\ep, \sigma, \nu, l \in \Qq_{>0}$ be fixed numbers. Let $\Ii \subset [0,1] \cap \Qq$ be a DCC set. Suppose that $\psi: (W, B_W) \to V$ is an $\ep$-lc Fano type fibration between projective normal varieties, and that there is a big and base point free Cartier divisor $R_V$ on $V$ such that they satisfy the following properties:
\begin{enumerate}
\item $\dim W =d$ and $B_W \in \Ii$,
\item $R_V^{\dim V} \leq \sigma$, and
\item in the canonical bundle formula $K_W+B_W\sim_\Qq \psi^*(K_V+\Delta_V+M_V)$, $(K_V+\Delta_V+M_V)\cdot R_V^{\dim V-1}<\nu$.
\end{enumerate} 
Then there exists $\mathfrak B = \mathfrak B(d, \ep, \sigma,\nu, l, \Ii) \in \Qq_{>0}$ such that
\[
\vol(K_W+2B_W+l\psi^*R_V) \leq \mathfrak B.
\]
\end{lemma}

\begin{proof}
We prove the lemma by induction on $\dim V$. First, we show the following claim.

\begin{claim*}\label{claim}
There exists $q=q(d, \ep, \sigma,\nu, l,\Ii) \in \Qq_{>0}$ such that $$ \vol(W, K_W+2B_W+l\psi^*R_V-q\psi^*R_V)=0.$$
\end{claim*}
\begin{proof}[Proof of the Claim]
This can be shown as \cite[Proposition 5.8]{Bir18}. In fact, suppose that there exists 
\[
0\leq \Theta \sim_\Qq K_W+2B_W+l\psi^*R_V-q\psi^*R_V,
\] then for a general fiber $F_\psi$ of $\psi: W \to V$, $$\Theta|_{F_\psi} \sim_\Qq K_{F_\psi}+2B_W|_{F_\psi} \sim_\Qq -K_{F_\psi}.$$ Hence by \cite[Theorem 1.6]{Bir16BAB}, there is $1>\tau=\tau(\ep,d)\in \Qq_{>0}$ such that $(F_\psi, B_W|_{F_\psi}+\tau \Theta|_{F_\psi})$ is klt. By inversion of adjunction, $(W, B_W+\tau \Theta)$ is klt over the generic point of $V$. Then
\[
K_W+(1-\tau)B_W+\tau \Theta = (1+\tau)(K_W+B_W+\frac{\tau(l-q)}{1+\tau}\psi^*R_V)\sim_\Qq 0/V.
\] By the canonical bundle formula,
\[
K_W+(1-\tau)B_W+\tau \Theta \sim_\Qq \psi^*(K_V+\Delta'_V+M_V')
\] with $\Delta'_V \geq 0$ and $M_V'$ pseudo-effective (\cite[Theorem 3.6]{Bir19}). It is well known that $K_V+(\dim V+2)R_V$ is big (for example, see \cite[Proposition 3.3]{Li18}). By
\[
\begin{split}
&(1+\tau)(K_W+B_W+\frac{\tau(l-q)}{1+\tau}\psi^*R_V) \\
\sim_\Qq& (1+\tau)\psi^*(K_V+\Delta_V+M_V+\frac{\tau(l-q)}{1+\tau}R_V)\\
\sim_\Qq&\psi^*(K_V+\Delta'_V+M_V'),
\end{split}
\] then $(1+\tau)(K_V+\Delta_V+M_V+\frac{\tau(l-q)}{1+\tau}R_V) + (d+2)R_V$ is big. Thus 
\begin{equation}\label{eq: q}
\begin{split}
0&\leq \big((1+\tau)(K_V+\Delta_V+M_V+\frac{\tau(l-q)}{1+\tau}R_V) + (d+2)R_V\big) \cdot R_V^{\dim V-1}\\
&= (1+\tau)(K_V+\Delta_V+M_V)\cdot R_V^{\dim V-1}+(\tau(l-q)+(d+2))\cdot R_V^{\dim V}\\
&\leq (1+\tau)\nu + (\tau(l-q)+(d+2))\cdot R_V^{\dim V}.
\end{split}
\end{equation}
Because $R_V$ is a nef Cartier divisor which is big, $R_V^{\dim V} \geq 1$. Then \eqref{eq: q} implies that when $q>l$, 
\[
0 \leq (1+\tau)\nu+(d+2) R_V^{\dim V}+\tau(l-q).
\] By $R_V^{\dim V} \leq \sigma$, we have
\[
q \leq \frac{(1+\tau)\nu+(d+2)\sigma}{\tau}+l.
\] Hence, if we choose $q=q(d, \ep, \sigma,\nu, l,\Ii) = \frac{(1+\tau)\nu+(d+2)\sigma}{\tau}+l+1$, then $ \vol(W, K_W+2B_W+l\psi^*R_V-q\psi^*R_V)=0$. 
\end{proof}

If $\dim V=0$, then 
\[
\vol(W, K_W+2B_W+l\psi^*R_V) = \vol(W, B_W)=\vol(W, -K_W).
\] Because $(W, B_W)$ is $\ep$-lc and $B_W$ is big, $W$ belongs to a bounded family which only depends on $\ep$ and $d$ (\cite[Corollary 1.2]{Bir16BAB}). Hence $\vol(W, -K_W)$ is bounded above.

\medskip

If $\dim V=1$, then take a general $S \in |\psi^*R_V|$. Because $\deg_V R_V \leq \sigma$, $S = \sqcup_{i=1}^\zeta S_i$ with $\zeta \leq \sigma$. Because $S$ is a disjoint union of prime divisors, Lemma \ref{le: jiang} still holds by the same argument. Thus, for any $q \in \Qq_{>0}$,
\begin{equation}\label{eq: dim 1}
\begin{split}
\vol(W, K_W+2B_W+l\psi^*R_V) \leq& \vol(W, K_W+2B_W+l\psi^*R_V-q\psi^*R_V)\\
&+qd\vol(S, (K_W+2B_W+l\psi^*R_V)|_{S}).
\end{split}
\end{equation} 
Taking $q=q(d, \ep,\sigma,\nu,l,\Ii)$ as in the above claim, by \eqref{eq: dim 1}, we have
\[
\begin{split}
\vol(W, K_W+2B_W+l\psi^*R_V) &\leq qd\vol(S, (K_W+2B_W+l\psi^*R_V)|_{S})\\ 
&\leq qd\sum_{i=1}^\zeta  \vol(S_i, K_{S_i}+2B_W|_{S_i})\\
&= qd\sum_{i=1}^\zeta  \vol(S_i, -K_{S_i}).
\end{split}
\] $\vol(S_i, -K_{S_i})$ is bounded above because $S_i$ belongs to a bounded family (see the $\dim V=0$ case). Then by $\zeta \leq \sigma$, we get the desired result.

\medskip

Finally, when $\dim V \geq 2$, choose a general $V_1 \in |R_V|$. $V_1$ is irreducible because $R_V$ defines a morphism to a variety with dimension $\geq 2$. Let $W_1= \psi^*V_1$ which is an irreducible divisor such that $(W, W_1+B_W)$ is plt, in particular, $W_1$ is normal. By adjunction, $K_{W_1}+B_{W_1} = (K_W+W_1+B_W)|_{W_1}$, $(W_1, B_{W_1})$ is still $\ep$-lc by inversion of adjunction (for example, see \cite[Corollary 1.4.5]{BCHM10}). Because $W_1$ is a general Cartier divisor, $B_{W_1}=B_W|_{W_1} \in \Ii$ (see \cite[Corollary 16.7]{Fli92}). Let $R_{V_1}=R_V|_{V_1}$ which is still big and base point free. Let $\psi_1: W_1 \to V_1$ be the induced morphism. Notice that a general fiber of $\psi_1$ is also a general fiber of $\psi$, hence $W_1$ is Fano type over $V_1$. Thus $\psi_1: W_1 \to V_1$ is an $\ep$-lc Fano type fibration. By canonical bundle formula, 
\[
K_{W_1}+B_{W_1} \sim_\Qq \psi_1^*(K_{V_1}+\Delta_{V_1}+M_{V_1})\sim_\Qq \psi_1^*((K_V+\Delta_V+M_V+R_V)|_{V_1}).
\] Hence, 
\begin{equation}\label{eq: induction bound}
\begin{split}
&(K_{V_1}+\Delta_{V_1}+M_{V_1})\cdot R_{V_1}^{\dim V_1 -1} \\
=&  (K_V+\Delta_V+M_V+R_V)\cdot V_1 \cdot R_{V}^{\dim V_1 -1}\\
= & (K_V+\Delta_V+M_V)\cdot R_{V}^{\dim V - 1} + R_{V}^{\dim V}\\
\leq & \nu + \sigma.
\end{split} 
\end{equation}

In summary, $\psi_1: (W_1, B_{W_1}) \to V_1$ satisfies the assumptions of Lemma \ref{le: for key prop} with $d, \ep,\sigma, \nu, \Ii$ replaced by $d-1, \ep, \sigma, \sigma+\nu, \Ii$. Hence, by the induction hypothesis, there exists $\mathfrak B_1 = \mathfrak B_1(d, \ep, \sigma,\nu, l, \Ii) \in \Qq_{>0}$ such that 
\[
\vol(W_1, K_{W_1}+2B_{W_1}+l\psi_1^*R_{V_1}) \leq \mathfrak B_1.
\] By Lemma \ref{le: jiang}, for any $q \in \Qq_{>0}$, we have
\[
\begin{split}
\vol(W, K_W+2B_W+l\psi^*R_V) \leq& \vol(W, K_W+2B_W+l\psi^*R_V-q\psi^*R_V)\\
&+qd\vol(W_1, K_{W_1}+2B_{W_1}+l\psi_1^*R_{V_1}).
\end{split}
\] Take $q=q(d, \ep,\sigma,\nu,l,\Ii) \in \Qq_{>0}$ as in the above claim. Then
\[
\vol(W, K_W+2B_W+l\psi^*R_V) \leq qd \mathfrak B_1,
\] and thus the induction step is completed.
\end{proof}

\begin{proposition}\label{prop: key}
Under the notation in the proof of Theorem \ref{thm: main result}, for any $l \in \Nn$, there exists $N=N(\Ii, \ep, d, v, l) \in \Rr_{>0}$ such that
\[
\vol(Y, K_Y+D_Y+lh^*A) < N.
\]
\end{proposition}

\begin{proof}
Because $(Y, D_Y)$ is a minimal model of $(X', D')/Z'$, $$\vol(Y, K_Y+D_Y+lh^*A) = \vol(X', K_{X'}+D'+lf'^*A).$$ Because $D'= (1+\ep/2)g_*^{-1}B+(1-\ep/2) \Exc(g), B' = g_*^{-1}B+(1-\ep/2) \Exc(g)$ and $\ep<1$, we have $D' \leq 2B'$. Hence $\vol(Y, K_Y+D_Y+lh^*A) \leq \vol(X', K_{X'}+2B'+lf'^*A)$. Recall that $\ti R = q^*A$, and because $K_W+2B_W={\pr_{2*}}\pr_1^*(K_{X'}+2B')$ where $X'\xleftarrow{\pr_1} X'' \xrightarrow{\pr_2} W$ is a common resolution, we have
\[
\vol(X', K_{X'}+2B'+lf'^*A) \leq \vol(W, K_W+2B_W+l\psi^*R_V).
\] Hence, it is enough to bound $\vol(W, K_W+2B_W+l\psi^*R_V)$. Notice that $R_V^{\dim V} \leq (mL_V)^{\dim V} \leq m^{\dim V}v$ where $m=m(\Ii,\ep,d,v)$, and $(K_V+B_V+M_V)\cdot R_V^{\dim V-1}<\mathfrak  R(\Ii, \ep, d, v)$ by Lemma \ref{le: key}. Applying Lemma \ref{le: for key prop} to $\psi: (W, B_W) \to V$, and taking $d, \ep,\sigma, \nu, l, \Ii$ to be $d, \ep, m^{\dim V}v, \mathfrak R, l, \Ii$ respectively, we have an upper bound for $\vol(W, K_W+2B_W+l\psi^*R_V)$  which depends only on $\Ii, \ep, d, v, l$.
\end{proof}

\begin{proof}[Proof of Corollary \ref{cor: bounded of the whole family}]
We show that $(X, B) \in \Ss(d,\ep, v,\Ii)$ is a $(d,r,\ep)$-Fano type fibration for some $d, r,\ep$. This certainly holds when $\dim Z=0$. Hence, we can assume $\dim Z>0$.

By the definition of $\Ss(d,\ep, v,\Ii)$, $(X, B) \to Z$ satisfies $(1), (2), (3)$ in Definition \ref{def: (d,r,ep)-Fano type fibration} with exactly the same $d$ and $\ep$. Now, we show (4) and (5). 

By the argument for the boundedness of $Z$ (see proof of Theorem \ref{thm: main result} Step 5), there are finitely many projective morphisms $\Xx^i \to \Tt$ and divisors $\Bb^{i'}$ on $\Xx^i$ whose coefficients vary in $[\delta, (1-\ep)]$ (but $\Supp \Bb^{i'}$ is fixed) satisfying the following properties: 

(1) there are finitely many birational maps $\varphi^i_{j}: \Xx^i \dasharrow \Yy^i_j/\Tt$ such that if $\varphi: (\Xx^i, \Bb^{i'}) \dasharrow (\Yy', \Bb'_{\Yy'})/\Tt$ is a log terminal model of $(\Xx^i, \Bb^{i'})/\Tt$, then there is a $j$ satisfying $\varphi^i_j=\varphi$ (this can be obtained by the same argument as Theorem \ref{thm: main result} Step 5 using \cite[Corollary 1.1.5 (1)]{BCHM10} instead of \cite[Corollary 1.1.5 (2)]{BCHM10}). 

(2) there are finitely many $\phi^i_{jk}: \Yy^{i}_j \to \mathcal Z_{jk}^i/\Tt$ which is a morphism to the ample model of some $(\Yy^{i}_j, \Bb^{i'}_{\Yy^i_j})/\Tt$. Hence it is also the ample model of some $(\Xx^{i}, \Bb^{i'})/\Tt$. Notice that $\phi^i_{jk} \circ \varphi_j^i: \Xx^i \dasharrow {\mathcal Z}_{jk}^i$ is exactly $\psi_i: \Xx \dasharrow \mathcal Z_j/\Tt$ in the proof of Theorem \ref{thm: main result} Step 5. 

For simplicity, we fix some $\Xx^i, \Yy^i_j, \mathcal Z^i_{jk}$, etc. and denote them by $\Xx, \Yy, \mathcal Z$, etc. By the construction of $\Bb'_{\Yy}$, it belongs to a rational polytope of divisors $\bar{\mathcal{P}}$ such that for any $\Dd \in \bar{\mathcal{P}}$, $K_{\Yy}+\Dd \sim_\Qq 0/\mathcal Z$ (this is \cite[Corollary 1.1.5 (2) (3)]{BCHM10} with $\bar \Aa_j$ replaced by $\bar{\mathcal{P}}$). Suppose that $\Dd^{(s)}$ corresponds to a vertex of $\bar{\mathcal{P}}$, and assume $K_{\Yy}+\Dd^{(s)} \sim_\Qq\phi^*\Ll^{(s)}$, where $\Ll^{(s)}$ is a divisor on $\mathcal Z$. Take a very ample divisor $\mathcal A$ on $\mathcal Z$ such that $\Aa - \Ll^{(s)}$ is ample for each $s$. Then for each $\Dd \in \bar{\mathcal{P}}$, $\Dd = \sum_s a^{(s)} \Dd^{(s)}$ with $a^{(s)} \geq 0$ and $\sum_s a^{(s)} =1$. Hence $K_{\Yy}+\Dd \sim_\Qq \phi^*(\sum_s a^{(s)} \Ll^{(s)})$, and $\Aa -(\sum_s a^{(s)} \Ll^{(s)})$ is still ample. Now suppose that $f: (X, B) \to Z$ is the Iitaka fibration in the definition of $\Ss(d, \ep,v,\Ii)$ with $K_X+B \sim_\Qq f^*L$. If $Z \simeq \mathcal Z_{t}$ for some $t \in \Tt$, then $B \sim_\Qq \Dd_t$ for some $\Dd \in \bar{\mathcal P}$. Hence $\Aa_t - L$ is ample. 

Because $\Aa_t^{\dim Z}$ can only achieve finitely many values, we can take $r$ to be the maximal value. Then for such $Z$, $A=\Aa_t$ and $r$ satisfy the requirements of (4) and (5) in Definition \ref{def: (d,r,ep)-Fano type fibration}. Because there are finitely many $\Xx, \Yy, \mathcal Z$, we can choose a uniform $r=r(d,\ep,v,\Ii)\in \Nn$ satisfying the requirements for each $Z$.

Finally, because $(X, B) \in \Ss(d,\ep, v,\Ii)$ satisfies the conditions in \cite[Theorem 1.3]{Bir18} (see Theorem \ref{thm: Birkar CY fibration}), $\Ss(d,\ep, v,\Ii)$ is log bounded.
\end{proof}

Corollary \ref{cor: DCC of Iitaka volume} follows from Corollary \ref{cor: bounded of the whole family}.

\begin{proof}[Proof of Corollary \ref{cor: DCC of Iitaka volume}]
Just as in Step 1 in the proof of Theorem \ref{thm: main result}, we can replace $(X, B)$ by its good minimal model and assume that $f: X \to Z$ is a morphism defined by the semi-ample divisor $K_X+B$. Suppose that $\{\Ivol(K_{X_i}+B_i)\}_{i\in\Nn}$ is a strictly decreasing sequence satisfying properties of Corollary \ref{cor: DCC of Iitaka volume}. Let $v\in \Qq_{>0}$ such that $\Ivol(K_{X_i}+B_i)<v$ for each $i$. Then by Corollary \ref{cor: bounded of the whole family}, $(X_i, B_i), {i\in\Nn}$ belong to a bounded family $(\Xx, \Bb) \to \Tt$. 

By taking an irreducible component of $\Tt$ and the corresponding $(\Xx, \Bb)$, we can assume that $\Xx, \Tt$ are irreducible. Moreover, we can assume that the points which parametrize $(X_i, B_i)$ are dense in $\Tt$. Take a log resolution $\pi: \Xx' \to \Xx$ of $(\Xx, \Bb)$ and set $\Bb' = \pi_*^{-1}\Bb+\Exc(\pi)$. For $(X_i, B_i)$ which is parametrized by a general point $t_i\in \Tt$ (i.e. $X_i \simeq \Xx_{t_i}$), set $X_i' = \Xx'_{t_i}$ and $B_i'=(\pi|_{\Xx'_{t_i}})_*^{-1}(B_i)+(1-\ep)(\Exc(\pi)_{t_i})$. Then 
\[
K_{X_i'}+B_i'=(\pi|_{\Xx'_{t_i}})^*(K_{X_i}+B_i)+E_i'
\] with $E_i'\geq 0$ a $\pi|_{\Xx'_{t_i}}$-exceptional divisor. Thus $\Ivol(K_{X_i'}+B_i') = \Ivol(K_{X_i}+B_i)$. Moreover, $(X_i', B_i')$ is $\ep$-lc and belongs to a bounded family $(\Xx', \Bb') \to \Tt$. Replacing $(\Xx, \Bb), (X_i, B_i)$ by $(\Xx', \Bb'), (X_i', B_i')$ respectively, and stratifying $\Tt$ if necessary, we can assume that $(\Xx, \Bb)$ is log smooth over $\Tt$, and $\Tt$ is smooth. Passing to a finite cover of $\Tt$ (see \cite[Claim 4.38.1]{Kol13}), we can further assume that every stratum of $\Bb$ has irreducible fibers over $\Tt$.

For $(X_i, B_i)$, let $\Bb^i$ be a divisor on $\Xx$ whose components choose the same coefficients as those of $B_i$. By \cite[Theorem 4.2]{HMX18}, for fixed $m \in \Nn$ and any $t\in \Tt$, $h^0(\Xx_t, m(K_{\Xx_t}+(\Bb^i)_t))$ is invariant. Because $\Ii$ is a DCC set, we can assume that the coefficient of each component of $\{\Bb^i\}_{i\in\Nn}$ is non-decreasing by passing to a subsequence. Moreover, we can assume that $\kappa(K_{X_i}+B_i) \geq 0$ is the same for each $i\in \Nn$. Then, for any $t\in \Tt$, $\{\Ivol(K_{\Xx_t}+(\Bb^i)_t)\}_{i\in\Nn}$ is non-decreasing. However, when $i<j$,
\[
\Ivol(K_{X_i}+B_i)=\Ivol(K_{\Xx_t}+(\Bb^i)_t) \leq \Ivol(K_{\Xx_t}+(\Bb^j)_t) = \Ivol(K_{X_j}+B_j),
\] which is a contradiction to the strictly decreasing of $\{\Ivol(K_{X_i}+B_i)\}_{i\in\Nn}$. 
\end{proof}

\begin{remark}\label{eq: rmk: discreteness of Ivol}
By the same argument as above, if further assume that $\Ii$ is a finite set, then $\{\Ivol(K_X+B)\}$ is a discrete set. 
\end{remark}

\section{Further discussions}\label{sec: further discussions}

In this section, we explain the relation between Conjecture \ref{conj: bounded of base, klt} and the effective adjunction conjecture (see Conjecture \ref{conj: effective adjunction}). Along the way, some lower dimensional cases are established. We also discuss possible variants of Conjecture \ref{conj: bounded of base, klt}.

\begin{proposition}\label{prop: assume effective adjunction}
Assuming Conjecture \ref{conj: effective adjunction} and the existence of good minimal models, then Conjecture \ref{conj: bounded of base, klt} and Conjecture \ref{conj: DCC} hold.
\end{proposition}
\begin{proof}
By the existence of good minimal models, we can assume that $K_X+B$ is semi-ample with $f: X \to Z$ the morphism induced by $K_X+B$.  By the canonical bundle formula, $K_X+B\sim_\Qq f^*(K_Z+\Delta_Z+M_Z)$ where $\Delta_Z$ belongs to a DCC set which depends only on $d$ and $\Ii$. For a general fiber $F$, $K_F+B|_F \equiv 0$ and $B|_F \in \Ii$. Hence by global ACC \cite[Theorem D]{HMX14}, $B|_F$ belongs to a finite set $\Ii_0$ depending on $\dim F$ and $\Ii$. Applying Conjecture \ref{conj: effective adjunction}, there is an $m=m(d, \Ii_0) \in \Nn_{\geq 2}$ such that for the moduli b-divisor $\mathbf M$, $m\mathbf M$ is a base point free b-divisor. Hence there exists $0 \leq G_Z \simeq_{\Qq} M_Z$ with $G_Z \in \{\frac{k}{m} \mid 1 \leq k \leq {m-1}, k \in \Zz\}$ such that $(Z, \Delta_Z+G_Z)$ is klt. Notice that $K_Z+\Delta_Z+G_Z$ is ample with $\vol(K_Z+\Delta_Z+G_Z) = \Ivol(K_X+B)$, and the coefficients of $\Delta_Z+G_Z$ belong to a DCC set depending only on $d$ and $\Ii$. Then Conjecture \ref{conj: bounded of base, klt}  follows from \cite[Theorem 1.1]{HMX18} and Conjecture \ref{conj: DCC} follows from \cite[Theorem 1.3 (1)]{HMX14}.
\end{proof}

\begin{lemma}[{\cite[Theorem 1.3]{TX09}}]\label{le: TX09}
Let $f: (X, B) \to Z$ be a klt-trivial fibration. When the relative dimension of $f$ is two, there is a natural number $m$ depending only on the coefficients of the horizontal part of $B$, such that for the moduli b-divisor $\mathbf M$, $m\mathbf M$ is b-Cartier.
\end{lemma}

\begin{remark}
The dependence only on the coefficients of the horizontal part of $B$ is not explicitly stated in \cite[Theorem 1.3]{TX09}, but it is easily seen from the proof of \cite[Theorem 1.3]{TX09} where $m$ depends only on $(F, B|_F)$ with $F$ a general fiber of $f$.
\end{remark}

\begin{corollary}\label{cor: dim 3}
Conjecture \ref{conj: bounded of base, klt} and Conjecture \ref{conj: DCC} hold when $d\leq 3$.
\end{corollary}
\begin{proof}
We only show the case $d=3$ as the easier case $d<3$ can be obtained by the same argument. When $d=3$, we only need to consider the cases when $\dim Z =1, 2$. 

If $\dim Z =2$, then the relative dimension is $1$, and Conjecture \ref{conj: effective adjunction} holds by \cite[Theorem 8.1]{PS09}. Hence the claim follows from Proposition \ref{prop: assume effective adjunction}. 

If $\dim Z=1$, then adopting the same notation as in the proof of Proposition \ref{prop: assume effective adjunction}, we have $\deg_Z K_Z \leq \deg_Z (K_Z+\Delta_Z+M_Z) =v$. Hence $Z$ belongs to a bounded family. For Conjecture \ref{conj: DCC}, by global ACC \cite[Theorem D]{HMX14}, $B|_F$ belongs to a finite set depending on $\dim F$ and $\Ii$, then by Lemma \ref{le: TX09}, there exists $m \in \Nn$ such that  $mM_Z$ is nef and Cartier. Thus $\deg_Z M_Z \in \frac{1}{m}\Zz_{\geq 0}$. Because $\Delta_Z$ belongs to a DCC set, $\Ivol(K_X+B)=\deg_Z(K_Z+\Delta_Z+M_Z)$ belongs to a DCC set.
\end{proof}

It is well known that in Conjecture \ref{conj: bounded of base, klt}, condition (2) cannot be replaced by $\Ivol(K_X+B) \leq v$. In fact, even when $K_X+\Delta$ is ample, we have the following unbounded family of varieties.

\begin{example}[{cf. \cite[Example 2.15]{Ale94}}]\label{eg: unbounded}
Let $X$ be a smooth surface and $\Delta =A+B$ be an snc divisor with reduced irreducible components $A$ and $B$ such that $A \cap B \neq \emptyset$. Suppose that $K_X+\Delta$ is ample. First, choose $0 < a_0 <1$ close to $1$, and then blow up an intersection point of $A\cap B$, we have a crepant model
$(X_1, a_0(A_1+B_1)+(2a_0-1)E)$, where $A_1, B_1$ are strict transforms of $A, B$. Choose $0<\epsilon_1 \ll 1$, then $K_{X_1} +\Delta_1\coloneqq K_{X_1}+a_0(A_1+B_1)+(2a_0-1-\epsilon_1)E)$ is ample. The coefficients of $\Delta_1$ are in $\{a_0, 2a_0-1-\epsilon\}$. Notice that $\vol(K_{X_1}+a_0(A_1+B_1)+(2a_0-1-\epsilon_1)E) \leq \vol(K_X+\Delta)$. Next, choose $a_1<1$ even more close to $1$. First blow up a point of $A\cap B$, then blow up a point of $A_1 \cap E$, we get a crepant model
\[
(X_2, a_1(A_2+B_2)+(2a_1-1)E+(3a_1-2)F). 
\] Take $0<\epsilon_2, \epsilon_2' \ll 1$ such that
\[
K_{X_2} +\Delta_2\coloneqq K_{X_2}+a_1(A_2+B_2)+(2a_1-1-\epsilon_2)E+(3a_1-2-\epsilon_2')F
\] is ample. Its volume is still bounded above by $\vol(K_X+\Delta)$. Besides,  $a_1, \epsilon_2, \epsilon_2'$ can be chosen such that $\min\{a_1, 2a_1-1-\epsilon_2, 3a_1-2-\epsilon_2'\} > \max\{a_0, 2a_0-1-\epsilon_1\}$. We can continue this process to construct klt log pairs $(X_i, \Delta_i), i \in \Nn$. The coefficients of $\Delta_i$ are in a DCC set and $\vol(K_{X_i}+\Delta_i) \leq \vol(K_X+\Delta)$, but $\{X_i\}$ is unbounded because $\rho(X_i)=\rho(X)+i$. 
\end{example}

However, under the $\ep$-lc assumption on the total space $(X, B)$, we propose the following conjecture.

\begin{conjecture}\label{conj: bounded of base, epsilon-lc}
Let $d\in\Nn, \epsilon, v \in \Qq_{>0}$ be fixed numbers, and $\Ii\subset [0,1] \cap \Qq$ be a DCC set. Let $\mathfrak F(d,\epsilon, v, \Ii)$ be the set of varieties $Z$ satisfying the following properties:
\begin{enumerate}
\item $(X, B)$ is $\ep$-lc with $\dim X =d$, and coefficients of $B$ are in $\Ii$,
\item $0<\Ivol(K_X+B) < v$, and
\item $f: X \dasharrow Z$ is the Iitaka fibration associated with $K_X+B$, where $$Z=\proj \oplus_{m=0}^\infty H^0(X, \Oo_X(\lfloor m(K_X+B) \rfloor)).$$
\end{enumerate}
Then $\mathfrak F(d,\epsilon, v, \Ii)$ is a bounded family.
\end{conjecture}

When $K_X+B$ is big, that is, $\Ivol(K_X+B)=\vol(K_X+B)>0$, Conjecture \ref{conj: bounded of base, epsilon-lc} follows from \cite[Theorem 1.3, Theorem 1.6]{HMX14}. Unfortunately, I do not know whether the conjecture holds when $\dim X=3, \dim Z =2$ and $B^h=0$. In this case, $(Z, B_Z)$ may not be $\delta$-lc for any $\delta>0$. Hence, a priori, $Z$ may not be in a bounded family (see Example \ref{eg: unbounded}). Our main result (Theorem \ref{thm: main result}) is about Conjecture \ref{conj: bounded of base, epsilon-lc} under the assumption that $-K_X$ is big over $Z$. Under this extra condition, $(Z, B_Z)$ is $\delta$-lc for some $\delta = \delta(\ep,d,\Ii)>0$ (see \cite[Theorem 1.9]{Bir18}).

\begin{remark}
As mentioned above, it is desirable to obtain the boundedness of the bases regardless of the boundedness of fibers. From this perspective, we can even ask whether $Z$ belongs to a bounded family under the assumption that $\dim Z$ is fixed in Conjecture \ref{conj: bounded of base, klt} and Conjecture \ref{conj: bounded of base, epsilon-lc} (i.e. $\dim X$ can be arbitrarily large).
\end{remark}

Besides, in Conjecture \ref{conj: bounded of base, klt} and Conjecture \ref{conj: bounded of base, epsilon-lc}, one can consider $\Rr$-divisors instead of $\Qq$-divisors. In this scenario, $\proj \oplus_{m=0}^\infty H^0(X, \Oo_X(\lfloor m(K_X+B) \rfloor))$ may not make sense. Instead, $Z$ should be replaced by the ample model of $(X, B)$ whose existence is known under certain conditions (see \cite{Jia22, Li22}). For Conjecture \ref{conj: DCC}, one can also require that $(X, B)$ to have lc singularities with real coefficients. See \cite[Section 6]{Li20} for the precise statement of the conjectures and their relation to the $\Gamma$-effective adjunction conjecture
for lc-trivial fibrations.

\bibliographystyle{alpha}

\bibliography{bibfile}

\end{document}